\pdfoutput=1
\RequirePackage{ifpdf}
\ifpdf 
\documentclass[pdftex]{sigma}
\else
\documentclass{sigma}
\fi

\numberwithin{equation}{section}

\usepackage{tikz}
\usetikzlibrary{positioning}

\newtheorem{thm}{Theorem}[section]

\newtheorem{lem}[thm]{Lemma}
\newtheorem{prop}[thm]{Proposition}
\newtheorem{introthm}[thm]{Theorem}

\newtheorem{introcor}[thm]{Corollary}

\theoremstyle{definition}
\newtheorem{dfn}[thm]{Definition}
\newtheorem{ex}[thm]{Example}
\newtheorem{rem}[thm]{Remark}

\newcommand\A{{\mathcal{A} }}
\newcommand\X{{\mathcal{X} }}

\newcommand\M{{\mathcal{M}}}

\newcommand\ve{{\varepsilon}}
\newcommand\bi{{\mathbf{s}}}
\newcommand{\bs}{\mathbf{s}}
\newcommand\bj{{\mathbf{t}}}
\newcommand\wG{{\widehat\Gamma}}
\newcommand\wE{{\widehat{\mathbb{E}}}}
\newcommand\wM{{\widehat{\mathcal{M}}}}
\newcommand\seq{{\sigma\mu_{i_k}\cdots\mu_{i_1}}}
\newcommand{\uf}{\mathrm{uf}}

\renewcommand\epsilon{{\varepsilon}}

\begin{document}
\allowdisplaybreaks

\newcommand{\arXivNumber}{1711.07785}

\renewcommand{\PaperNumber}{025}

\FirstPageHeading

\ShortArticleName{Presentations of Cluster Modular Groups and Generation by Cluster Dehn Twists}

\ArticleName{Presentations of Cluster Modular Groups\\ and Generation by Cluster Dehn Twists}

\Author{Tsukasa ISHIBASHI}

\AuthorNameForHeading{T.~Ishibashi}

\Address{Graduate School of Mathematical Sciences, The University of Tokyo,\\ 3-8-1 Komaba, Meguro, Tokyo 153-8914, Japan}
\Email{\href{mailto:ishiba@ms.u-tokyo.ac.jp}{ishiba@ms.u-tokyo.ac.jp}}

\ArticleDates{Received January 01, 2020, in final form March 27, 2020; Published online April 07, 2020}

\Abstract{We give a method to compute presentations of saturated cluster modular groups. Using this, we obtain finite presentations of the saturated cluster modular groups of finite mutation type $X_6$ and $X_7$. We verify that the cluster modular groups of finite mutation type $\widetilde{E}_6$, $\widetilde{E}_7$, $\widetilde{E}_8$, $G_2^{(*,*)}$, $X_6$ and $X_7$ are virtually generated by cluster Dehn twists.}


\Classification{13F60; 05E15; 30F60}

\section{Introduction and main results}

A \emph{cluster modular group}, defined in~\cite{FG09}, is a group $\Gamma_{|\bi|}$ associated with a combinatorial data $\bi$ called a \emph{seed}. An element of the cluster modular group is a finite sequence of \emph{seed permutations} and \emph{mutations} which preserve the \emph{exchange matrix}, and two such sequences are identified if they induce the same pair of \emph{cluster transformations}. The cluster modular group acts on the \emph{cluster modular complex} $\M_{|\bi|}$. The quotient $\M_{|\bi|}/\Gamma_{|\bi|}$ is called the \emph{modular orbifold}, which can be considered as a combinatorial generalization of the moduli space of Riemann surfaces. Indeed, the modular orbifold coincides with the latter space when the seed is associated with an ideal triangulation of a closed surface with one puncture~\cite{Penner}. Therefore the structure of the cluster modular orbifold is of great interst, especially in the context of the \emph{higher Teichm\"uller theory}~\cite{FG06}. For example, a fundamental problem is to compute the rational cohomology groups of the modular orbifold, which coincide with those of its orbifold fundamental group. A presentation of the cluster modular group will provide useful information for the latter.

Once trying to find a presentation of the cluster modular group, one immediately encounters the difficulty which arises from the fact that a complete list of relations among the cluster transformations is not known in general. In simple cases they are exhaused by \emph{standard $(h+2)$-gon relations}~\cite{FG08} such as the involutivity and the \emph{pentagon relation}, while there are ``non-standard'' relations in general, even for those associated with marked surfaces \cite{FST08}. A nice survey on this problem is found in~\cite{Kim-Yamazaki}: not only an annoying thing is this, but also related to certain ``dualities'' between supersymmetric gauge theories.

In order to isolate such a problem, we consider the \emph{saturated cluster modular group} $\widehat{\Gamma}_{\bs}$ \cite{FG08} instead. It is defined by restricting the relations among cluster transformations to those generated by standard ones, so that the cluster modular group is obtained as a quotient of the saturated cluster modular group.
The saturated cluster modular group is basically easier than the cluster modular group to deal with, and already considered in several studies:
\begin{itemize}\itemsep=0pt
\item The fundamental group of the modular orbifold $\M_{|\bi|}/\Gamma_{|\bi|}$ is actually isomorphic to the saturated cluster modular group~\cite{FG09}.
\item The Fock--Goncharov quantization provides a projective unitary representation of the saturated cluster modular group~\cite[Theorem~5.5]{FG08}.
\item The Kato--Terashima partition $q$-series~\cite{KT15} gives a map from the saturated cluster modular group to the ring $\mathbb{Z}\big[\big[q^{1/\Delta}\big]\big]$ of formal power series for some integer~$\Delta$.
\end{itemize}
In this paper, we introduce a simplicial complex called the \emph{saturated modular complex} on which the saturated cluster modular group faithfully acts, and show that it is simply-connected. Then we can utilize a method established by Brown~\cite{Brown84} to obtain a presentation of the saturated cluster modular group from the data of this action. When the seed is of \emph{finite mutation type}, namely the mutation class of the exchange matrix is a finite set, this method works particularly well. In this case
the ``fundamental domain'' of the saturated modular complex is finite, and we can obtain a finite presentation of the saturated cluster modular group. The mutation classes of finite mutation type has been completely classified in~\cite{FeST, FeST12}: see Theorems~\ref{thm:FSTskewsym} and~\ref{thm:FSTgeneral}. In the case of skew-symmetric exchange matrices, the list consists of the mutation classes associated with marked surfaces, several classes associated with generalized Dynkin diagrams, and two mysterious classes called~$X_6$ and~$X_7$. The initial quivers of type~$X_6$ and~$X_7$ are shown in Fig.~\ref{fig: X_6 and X_7}. Our main result is a computation of finite presentations of the saturated cluster modular groups of type~$X_6$ and~$X_7$.
\begin{figure}[t]\centering
\scalebox{0.9}{
\begin{tikzpicture}
\begin{scope}[>=latex]
\fill (0,0) circle(2pt) coordinate(A1) node[above]{$1$};
\fill (A1) ++(165: 2) circle(2pt) coordinate(B1) node[left]{$2$};
\fill (A1) ++(135: 2) circle(2pt) coordinate(C1) node[left]{$3$};
\fill (A1) ++(45: 2) circle(2pt) coordinate(D1) node[right]{$4$};
\fill (A1) ++(15: 2) circle(2pt) coordinate(E1) node[right]{$5$};
\fill (A1) ++(270: 2) circle(2pt) coordinate(F1) node[below]{$6$};

\draw[->,shorten >=2pt,shorten <=2pt] (A1) -- (B1) [thick];
\draw[->,double,shorten >=2pt,shorten <=2pt] (B1) -- (C1) [thick];
\draw[->,shorten >=2pt,shorten <=2pt] (C1) -- (A1) [thick];
\draw[->,shorten >=2pt,shorten <=2pt] (A1) -- (D1) [thick];
\draw[->,double,shorten >=2pt,shorten <=2pt] (D1) -- (E1) [thick];
\draw[->,shorten >=2pt,shorten <=2pt] (E1) -- (A1) [thick];
\draw[->,shorten >=2pt,shorten <=2pt] (A1) -- (F1) [thick];

\fill (6,0) circle(2pt) coordinate(A) node[above]{$0$};
\fill (A) ++(165: 2) circle(2pt) coordinate(B) node[left]{$1$};
\fill (A) ++(135: 2) circle(2pt) coordinate(C) node[left]{$2$};
\fill (A) ++(45: 2) circle(2pt) coordinate(D) node[right]{$3$};
\fill (A) ++(15: 2) circle(2pt) coordinate(E) node[right]{$4$};
\fill (A) ++(285: 2) circle(2pt) coordinate(F) node[below]{$5$};
\fill (A) ++(255: 2) circle(2pt) coordinate(G) node[below]{$6$};

\draw[->,shorten >=2pt,shorten <=2pt] (A) -- (B) [thick];
\draw[->,double,shorten >=2pt,shorten <=2pt] (B) -- (C) [thick];
\draw[->,shorten >=2pt,shorten <=2pt] (C) -- (A) [thick];
\draw[->,shorten >=2pt,shorten <=2pt] (A) -- (D) [thick];
\draw[->,double,shorten >=2pt,shorten <=2pt] (D) -- (E) [thick];
\draw[->,shorten >=2pt,shorten <=2pt] (E) -- (A) [thick];
\draw[->,shorten >=2pt,shorten <=2pt] (A) -- (F) [thick];
\draw[->,double,shorten >=2pt,shorten <=2pt] (F) -- (G) [thick];
\draw[->,shorten >=2pt,shorten <=2pt] (G) -- (A) [thick];

\end{scope}
\end{tikzpicture}
}
\caption{Quivers of type $X_6$ and $X_7$.}\label{fig: X_6 and X_7}
\end{figure}

\begin{introthm}\label{introthm: X_7}
The saturated cluster modular group $\wG_{X_7}$ of type $X_7$ is generated by elements $\psi_k$, $\phi_k$ for $k=1,3,5$ and the permutation group $\mathcal{S}_3$ of numbers $\{1,3,5\}$, and the complete set of relations among them is given as follows:
\begin{gather*}
\sigma\psi_k\sigma^{-1}=\psi_{\sigma(k)}, \qquad
\sigma\phi_k\sigma^{-1}=\phi_{\sigma(k)}\qquad \text{for $\sigma \in \mathcal{S}_3$, $k=1,3,5$}, \\
\phi_1\phi_3=\phi_3\phi_1, \qquad
\phi_1=\psi_1^2\sigma_{35},\qquad
1=\big(\psi_3^{-1}\psi_1\big)^2,\qquad
1=\sigma_{153}\psi_5\psi_1\psi_3\psi_5\psi_1,
\end{gather*}
and the usual relations of permutations. Here $\sigma_{ij}$ denotes the tranposition of $i$ and $j$ and $\sigma_{153}$ denotes the cyclic permutation $1 \mapsto 5 \mapsto 3 \mapsto 1$.
\end{introthm}
Some of these relations have geometric interpretations, which are studied in Appendix~\ref{mapping class groups}.

\begin{introthm}\label{introthm: X_6}
The saturated cluster modular group $\wG_{X_6}$ of type $X_6$ is generated by five elements $\alpha_1$, $\alpha_2$, $\beta_1$, $\beta_2$ and $\sigma$, and the complete set of relations among them is given as follows:
\begin{gather*}
 \sigma^2=1, \qquad
 \alpha_2 = \sigma \alpha_1 \sigma^{-1}, \qquad \beta_2 = \sigma \beta_1 \sigma^{-1}, \\
 \alpha_1\alpha_2 = \alpha_2\alpha_1, \qquad
 \beta_1^{-1}\alpha_2\beta_1 = \alpha_1, \qquad
 \big(\beta_2\beta_1^{-1}\big)^2 = 1, \\
 \big(\alpha_2\beta_2\beta_1^{-1}\alpha_1^{-1}\big)^3 = 1,\qquad
 \beta_2 (\beta_1\alpha_1)^{-1}\beta_2 = \operatorname{Ad}_{\alpha_2\beta_1} \sigma, \qquad
 \beta_1 = \alpha_1 \big(\beta_1 \alpha_2 \beta_1^{-1}\big) \alpha_2^{-1}\sigma^{-1}.
\end{gather*}
Here $\operatorname{Ad}_x y:=xyx^{-1}$ denotes the conjugation.
\end{introthm}

\begin{rem}[the group $\wG_{X_7}$ as an amalgamated product\footnote{The author thank an anonymous referee for their suggestion to get this kind of presentation.}]
From the presentation given in Theorem~\ref{introthm: X_7}, we can delete the generators $\phi_k$ for $k=1,3,5$ using the fourth relation.
Let $G$ be the group generated by the three elements $\psi_k$ for $k=1,3,5$ subject to the following relations:
\begin{gather}
 \psi_i\psi_j^{-1} = \psi_j\psi_i^{-1},\label{eq:amal_quasi_comm} \\
 \psi_k^2 [\psi_i,\psi_j] \psi_k^{-2} = \psi_k\psi_i\psi_j\psi_k\psi_i, \label{eq:amal_commutator}
\end{gather}
for $\{i,j,k\}=\{1,3,5\}$. Then we have an isomorphism $\mathcal{S}_3 \ast_{\mathbb{Z}/3} G \cong \wG_{X_7}$, where the amalgamation data is given by
\begin{alignat*}{3}
 & \mathbb{Z}/3 \to \mathcal{S}_3, \qquad && \overline{1} \mapsto \sigma_{153}, &\\
 & \mathbb{Z}/3 \to G, \qquad && \overline{1} \mapsto \psi_5\psi_1\psi_3\psi_5\psi_1.&
\end{alignat*}
Indeed, using the first relation in Theorem~\ref{introthm: X_7} we can move all the generators from $\mathcal{S}_3$ to the left and get $\wG_{X_7}=\mathcal{S}_3\cdot \langle \psi_1,\psi_3,\psi_5\rangle$. The fifth relation is equivalent to $\psi_1\psi_3^{-1} = \psi_3\psi_1^{-1}$, which holds among the $\psi$'s and leads to the defining relation~\eqref{eq:amal_quasi_comm}. The third relation (``the $\phi$'s commute with each other'') is rewritten as
\begin{gather*}
 \sigma_{153} = \psi_5^2\psi_1^2\psi_3^{-2}\psi_5^{-2}
 = \psi_5^2\psi_1\psi_3\psi_1^{-1}\psi_3^{-1}\psi_5^{-2}
 =\psi_5^2[\psi_1,\psi_3]\psi_5^{-2},
\end{gather*}
which leads to \eqref{eq:amal_commutator} and describes the intersection of $\mathcal{S}_3$ and $\langle \psi_1,\psi_3,\psi_5\rangle$, together with the sixth relation.
\end{rem}

{\bf First homology groups.}
As the first application, we compute the first homology groups of~$\wG_{X_6}$ and~$\wG_{X_7}$. Here the first homology group ($=$ abelianization) of a group~$G$ is defined to be $H_1(G; \mathbb{Z})=G/[G,G]$. Here we present the results with proofs based on Theorems~\ref{introthm: X_7} and~\ref{introthm: X_6}.

\begin{introcor}\label{X_7 abelian}
We have $H_1\big(\wG_{X_7}; \mathbb{Z}\big) \cong \mathbb{Z}/5 \times \mathbb{Z}/2$. The generators are the images of~$\psi_1$ and~$\sigma_{13}$.
\end{introcor}

\begin{proof}
It is well-known that the signature function gives an isomorphism $H_1(\mathfrak{S}_n; \mathbb{Z}) \xrightarrow{\sim} \mathbb{Z}/2$ for the symmetric group $\mathfrak{S}_n$ of degree $n$. From the first relation in Theorem~\ref{introthm: X_7} we get $\psi_1=\psi_3=\psi_5$ in the abelianization, and the last relation implies $\psi_1^5=1$.
\end{proof}

\begin{introcor}\label{X_6 abelian}
We have $H_1\big(\wG_{X_6}; \mathbb{Z}\big) \cong \mathbb{Z} \times \mathbb{Z}/2$. The generators are the images of $\alpha_1$ and $\beta_1$.
\end{introcor}

\begin{proof}
From the second and third relations in Theorem~\ref{introthm: X_6}, we get $\alpha_1=\alpha_2$ and $\beta_1=\beta_2$ in the abelianization. The last two relations together with $\sigma^2=1$ imply that $\sigma = \alpha_1^{-1}\beta_1$ and $\alpha_1^2 = \beta_1^2$.
\end{proof}

{\bf Cluster Dehn twists.}
Next let us turn our attention to the problem finding a good generators of cluster modular groups.
As a candidate for an appropriate class of generators, the \emph{cluster Dehn twists} has been introduced in the author's previous work~\cite{Ish}. He proved that the cluster Dehn twists have a similar dynamical behavior to that of Dehn twists in mapping class groups. In the case of mapping class groups, Dehn twists and half-twists are cluster Dehn twists.

It is a classical theorem in the Teichm\"uller theory that mapping class groups of marked surfaces are generated by Dehn twists and half-twists. See, for instance, \cite{FM}. The following is a~cluster algebraic generalization of this theorem, for several cases of finite mutation type:

\begin{introthm}\label{introthm: generation}
The cluster modular groups of finite mutation type $\widetilde{E}_6$, $\widetilde{E}_7$, $\widetilde{E}_8$, $G_2^{(*,*)}$ and $X_7$ are generated by finitely many cluster Dehn twists. The cluster modular group of type $X_6$ is virtually generated by four cluster Dehn twists.
\end{introthm}

For the former three cases, the theorem follows from the computation of the cluster modular groups given by Assem--Schiffler--Shramchenko~\cite{ASS} using the cluster categories. The saturated cluster modular group of type $G_2^{(*,*)}$ is computed by Fock--Goncharov~\cite{FG06a}. For the last two cases we use Theorems~\ref{introthm: X_7} and~\ref{introthm: X_6}.

We can also find cluster Dehn twists in the remaining cluster modular groups of finite mutation type, at least for skew-symmetric cases. Our general expectation is that any cluster modular group of finite mutation type is virtually generated by cluster Dehn twists. It will be especially interesting to study the cases $E_7^{(1,1)}$ and $E_8^{(1,1)}$, since they appear as the unfrozen parts of the quivers defining the higher Teichm\"uller theory~\cite{FG06} for certain polygon (as a marked surface) and a Lie algebra of type $A$. See Table~\ref{tab:finite mutation from higher Teich}. Their mutation classes consist of $506$ and $5739$ quivers, respectively.

\begin{table}[h]\centering
 \begin{tabular}{c|lllll}
 & 3-gon & 4-gon & 5-gon & 6-gon & 7-gon \\ \hline
 $\mathfrak{sl}_2$ & $\varnothing$ & $A_1$ & $A_2$ & $A_3$ & $A_4$ \\
 $\mathfrak{sl}_3$ & $A_1$ & $D_4$ & $E_7$ & $E_8^{(1,1)}$ & $\infty_{13}$ \\
 $\mathfrak{sl}_4$ & $A_3$ & $E_7^{(1,1)}$ & $\infty_{15}$ & $\infty_{21}$ & $\infty_{27}$ \\
 $\mathfrak{sl}_5$ & $D_6$ & $\infty_{16}$ & $\infty_{26}$ & $\infty_{36}$ & $\infty_{46}$ \\
 $\mathfrak{sl}_6$ & $E_8^{(1,1)}$ & $\infty_{25}$ & $\infty_{40}$ & $\infty_{55}$ & $\infty_{70}$ \\
 $\mathfrak{sl}_7$ & $\infty_{15}$ & $\infty_{36}$ & $\infty_{57}$ & $\infty_{78}$ & $\infty_{99}$ \\
 \end{tabular}
\caption{Type of the unfrozen part of the quiver defining the higher Teichm\"uller theory for the $(m+2)$-gon and the Lie algebra $\mathfrak{sl}_{n+1}$ with $m=1,\dots,5$ and $n=1,\dots,6$. Here $\varnothing$ denotes the empty quiver, and~$\infty_N$ denotes a mutation class of a quiver of infinite mutation type with~$N$ vertices.}\label{tab:finite mutation from higher Teich}
\end{table}

{\bf Organization of the paper.} In Section~\ref{section: definitions}, we recall the definition of the saturated cluster modular groups. We introduce the saturated modular complexes and prove that they are simply-connected. In Section~\ref{section: Brown}, we recall Brown's algorithm~\cite{Brown84} which enables us to compute a~presentation of a group acting on a CW complex. We investigate relevant properties of the action of the saturated cluster modular group on the saturated modular complex in full generality. In Section~\ref{section: proof}, we compute the presentations of the saturated cluster modular groups of type~$X_7$ and~$X_6$. In Section~\ref{sec:generation}, we give a proof of Theorem~\ref{introthm: generation}.

\section{Basic definitions}\label{section: definitions}

\subsection{Seed mutations}
In this section we review seed mutations, in order to motivate the definitions of the saturated cluster modular group and the saturated modular complex.
We follow the notations of~\cite{FG09}.

Let $I$ be a finite set.
A \emph{seed} is a tuple $\bi=(\ve,(A_i)_{i \in I}, (X_i)_{i \in I})$, where $\ve=(\ve_{ij})_{i,j \in I}$ is a $\frac{1}{2}\mathbb{Z}$-valued skew-symmetrizable matrix, $(A_i)_{i \in I}$ and $(X_i)_{i \in I}$ are two bunches of algebraically independent commutative variables. Our convention of skew-symmetrization is that there exist positive integers $d_j$ ($j \in I$) such that the matrix $\big(\ve_{ij}d_j^{-1}\big)_{i,j}$ is skew-symmetric. The matrix~$\ve$ is called the \emph{exchange matrix} and the variables $(A_i)$ (resp.~$(X_i)$) are called the cluster $\A$- (resp.~$\X$-)variables.
We fix a subset $I_0 \subset I$ called the \emph{frozen subset}, and assume $\ve_{ij} \in \mathbb{Z}$ unless $(i,j) \in I_0 \times I_0$.

The exchange matrix $\ve$ can be represented by a weighted quiver without loops and 2-cycles. The corresponding quiver has the set of vertices $I$ and each vertex $i \in I$ is assigned a weight~$d_i$, and there exist $\big|d_j^{-1}\ve_{ij}\big|\gcd (d_i,d_j)$ arrows from the vertex $i$ to the vertex $j$ (resp.~$j$ to~$i$) if $\epsilon_{ij}>0$ (resp.~$\epsilon_{ij}<0$). See~\cite{Zickert} for details. We tacitly use this correspondence in the sequel. In this sense we call an element of $I$ a \emph{vertex} and elements of~$I_0$ are called \emph{frozen vertices}.

\begin{dfn}[seed mutations]
For a seed $\bi=(\ve,(A_i)_{i \in I}, (X_i)_{i \in I})$ and a vertex $k \in I \setminus I_0$, we define a new seed $\bi'=(\ve',(A'_i)_{i \in I}, (X'_i)_{i \in I})$ as follows:
\begin{gather}
\ve'_{ij}:=
\begin{cases}
-\ve_{ij} & \text{if $k \in \{ i, j\}$}, \vspace{2mm} \\
\ve_{ij} + \dfrac{|\ve_{ik}|\ve_{kj}+ \ve_{ik}|\ve_{kj}|}{2} & \text{ otherwise},
\end{cases} \label{e-mutation} \\
A_i':=
\begin{cases}
\displaystyle A_i^{-1}\bigg(\prod_{\ve_{kj}>0} A_j^{\ve_{kj}} +
 \prod_{\ve_{kj}<0} A_j^{-\ve_{kj}}\bigg) & \text{if $i=k$}, \\
A_i & \text{otherwise},
\end{cases}\label{A-mutation} \\
X_i':=
\begin{cases}
X_k^{-1} & \text{if $i=k$}, \\
X_i\big(1+ X_k^{-\operatorname{sgn}(\ve_{ik})}\big)^{-\ve_{ik}} & \text{otherwise}.
\end{cases} \label{X-mutation}
\end{gather}
We write $\bi' = \mu_k(\bi)$ and refer to this transformation of seeds as the \emph{mutation directed to the vertex~$k$}. The transformation~\eqref{e-mutation} is independent of the other data and called the \emph{matrix mutation} (or \emph{quiver mutation}). The rational transformations~\eqref{A-mutation} and~\eqref{X-mutation} are called the \emph{cluster $\A$-} and \emph{$\X$-transformations}.
\end{dfn}

A \emph{seed permutation} is a permutation $\sigma$ of $I$ which preserves $I_0$ setwise. It acts on a seed as $\sigma(\ve,(A_i),(X_i)) = (\ve',(A'_i),(X'_i))$, where
\[
\ve'_{ij}:=\ve_{\sigma^{-1}(i),\sigma^{-1}(j)},\qquad A'_i:=A_{\sigma^{-1}(i)},\qquad X'_i:=X_{\sigma^{-1}(i)}.
\]
It is called a \emph{seed isomorphism} if it satisfies $\ve'=\ve$. A \emph{mutation sequence} is a finite composition of mutations and seed permutations. A mutation sequence is called a \emph{mutation loop} if it preserves the initial exchange matrix.
It is said to be \emph{trivial} if it also preserves the initial cluster variables.
The set
\begin{align*}
 |\bi|:=\{ \phi(\bi) \,|\, \text{$\phi$: mutation sequence with the initial seed $\bi$} \}
\end{align*}
is called the \emph{mutation class} of $\bi$. The set $|\ve|$ of exchange matrices (or weighted quivers) appearing in $|\bi|$ is called the mutation class of $\ve$.

The \emph{cluster modular group} $\Gamma_\bi$ based at a seed $\bi$ is the group of mutation loops with the initial seed $\bi$, modulo trivial ones. Since the group structure of the cluster modular group only depends on the mutation class $|\bi|$, we will use the notation $\Gamma_{|\bi|}$ when no confusion can occur.

\begin{rem}
The cluster modular group acts on some geometric objects $\A_{|\bi|}$ and $\X_{|\bi|}$ called the \emph{cluster $\A$- and $\X$-varieties}, preserving their geometric structures. See~\cite{FG09} for details.
\end{rem}

\begin{ex}\label{ex:MCG}
Here is a geometric example. Let $\Sigma$ be a marked surface. It is a connected oriented compact 2-dimensional manifold with boundary equipped with a finite subset $M \subset \Sigma$ of marked points, satisfying some conditions. See \cite{FeST} for details. An ideal triangulation of $\Sigma$ is the isotopy class of a collection $\Delta=\{\alpha_i\}_{i=1}^{n(\Sigma)}$ of simple arcs whose endpoints are marked points, and the complement $\Sigma \setminus \bigcup_{i=1}^{n(\Sigma)}\alpha_i$ consists of triangles. Here $n(\Sigma)$ only depends on $\Sigma$.

Given such an ideal triangulation $\Delta$, we draw a quiver on each triangle as shown in Fig.~\ref{fig:triangle}. Gluing them along the edges, we get a quiver drawn on a surface. The vertices on $\partial \Sigma$ are declared to be frozen. For example, a torus with one marked point and its ideal triangulation yields a quiver whose exchange matrix is given by
\[
\ve=\begin{pmatrix}
0 & 2 & -2 \\
-2& 0 & 2 \\
2 & -2 & 0
\end{pmatrix}.
\]
Hence we can form a seed $\bi_\Delta$. Here cluster variables can be interpreted as coordinate functions on certain extensions of the Teichm\"uller space of $\Sigma$. See~\cite{Penner} for details. Except for a few small surfaces, we have
\[
\Gamma_{\bi_\Delta} \cong
	\begin{cases}
		{\rm MC}(\Sigma) & \text{if $\Sigma$ is a closed surface with one marked point},\\
		{\rm MC}(\Sigma) \ltimes \{\pm 1\}^p & \text{otherwise}.
	\end{cases}
\]
Here $p$ is the number of interior marked points (punctures) and ${\rm MC}(\Sigma)$ denotes the \emph{mapping class group} of $\Sigma$, which is the group of isotopy classes of orientation-preserving diffeomorphisms on $\Sigma$ that fix the subset $M$ setwise.
Its action on $\{\pm 1\}^p$ is induced by the permutation of punctures. See \cite{BS15,Ish}.
\end{ex}

\begin{figure}[t]\centering
\begin{tikzpicture}
\draw (0,0) to (2,0) to (1, 1.732) to (0,0);
{\color{red}
	\draw (1,0) circle(2pt) coordinate(A);
	\draw (1.5, 0.866) circle(2pt) coordinate(B);
	\draw (0.5, 0.866) circle(2pt) coordinate(C);
	\begin{scope}[>=latex]
		\draw[->,thick] (A) to (B);
		\draw[->,thick] (B) to (C);
		\draw[->,thick] (C) to (A);
	\end{scope}
}
\end{tikzpicture}
\caption{The quiver associated with a triangle.}\label{fig:triangle}
\end{figure}

The definition of the cluster modular group contains an equivalence relation which comes from those among mutations. In general it is difficult to list up the generators of relations, while we know the following ``standard'' ones:

\begin{lem}[standard $(h+2)$-gon relations~\cite{FG09}]\label{standard relations}
Let $\bi=(\ve,(A_i)_{i \in I}, (X_i)_{i \in I})$ be a seed and $(p,h)=(0,2)$, $(1,3)$, $(2,4)$ or $(3,6)$. Then for each $i, j \in I$ such that $\ve_{ij}=-p \ve_{ji}=p$, the mutation sequence $r_{ij}:=((i\ j)\mu_i)^{h+2}=((i\ j)\mu_j)^{h+2}$ is trivial. Here $(i\ j)$ denotes the seed permutation given by the transposition of $i$ and $j$.
\end{lem}
We call these sequences $r_{ij}(=r_{ji})$ the \emph{standard sequences}. Note that if the exchange matrix is skew-symmetric, which is the case we treat mainly in this paper, we have only two of the standard relations:
\begin{itemize}\itemsep=0pt
\item $\mu_i\mu_j=\mu_j\mu_i$ if $\ve_{ij}=0$ (square relation).
\item $\mu_i\mu_j\mu_i\mu_j\mu_i=(i\ j)$ if $\ve_{ij}=\pm 1$ (pentagon relation).
\end{itemize}

\begin{ex}Here is a basic example of the pentagon relation in cluster modular group.
Let $I:=\{0,1\}$, $I_0:=\varnothing$ and consider a skew-symmetric exchange matrix
\[
\ve:= \begin{pmatrix}
0 & 1 \\ -1 & 0
 \end{pmatrix},
\]
which is called the exchange matrix of type $A_2$.
The mutation sequence $\phi:=(0\ 1)\mu_0$ gives an element of the cluster modular group. It turns out that it is the generator of the cluster modular group. The associated cluster transformations are given by
\begin{gather*}
\phi(A_0, A_1)=\left(A_1, \frac{1+A_1}{A_0}\right), \qquad
\phi(X_0, X_1)=\big(X_1(1+X_0), X_0^{-1}\big).
\end{gather*}
Then the pentagon relation implies that $\phi$ has order $5$, or one can check it by seeing the cluster transformations. In particular we have $\Gamma_{A_2} \cong \mathbb{Z}\slash 5$.
\end{ex}

\paragraph{\textbf{Cluster Dehn twists.}}
An element $\phi \in \Gamma_{|\bi|}$ of infinite order is called a \emph{cluster Dehn twist} if there exists a seed $\bj \in |\bi|$ such that $\phi^n=((i\ j)\mu_j)^\ell$ as an element of $\Gamma_\bj$ for some non-zero integers~$n$,~$\ell$ and vertices $i,j \in I-I_0$.
For example, $I:=\{0,1\}$, $I_0:=\varnothing$ and consider a~skew-symmetric exchange matrix
\[
\ve:= \begin{pmatrix}
0 & k \\ -k & 0
 \end{pmatrix}
\]
for an integer $k \geq 2$. Then the mutation sequence $\phi:=(0\ 1)\mu_0$ gives an element of the cluster modular group, and it is a cluster Dehn twist.

\subsection{Saturated cluster modular groups and saturated modular complexes}
In this section we recall the definition of the \emph{saturated cluster modular group}~\cite{FG08} (or the \emph{special cluster modular group}~\cite{FG09}). Also we introduce the \emph{saturated cluster complex} on which the saturated cluster modular group acts.

\begin{dfn}The \emph{saturated cluster modular group} $\widehat{\Gamma}_\ve$ based at an exchange matrix $\ve$ is the group of (matrix) mutation loops with the initial exchange matrix $\ve$, modulo the equivalence relation generated by:
\begin{enumerate}\itemsep=0pt
\item Involutivity: $\mu_k\mu_k=1$ for each $k \in I$,
\item Naturality: $\sigma\mu_k\sigma^{-1}=\mu_{\sigma(k)}$ for each $k \in I$ and a seed permutation $\sigma$, and
\item Standard relations: $r_{ij}=r_{ji}=1$ under the situation of Lemma~\ref{standard relations}.
\end{enumerate}
\end{dfn}
Take a seed $\bi=(\ve,(A_i)_{i \in I}, (X_i)_{i \in I})$ with the exchange matrix $\ve$.
Then Lemma~\ref{standard relations} implies that we have a surjective homomorphism $\widehat\Gamma_\ve \to \Gamma_\bi$.

To investigate the saturated cluster modular group, we define a 2-complex on which the saturated cluster modular group acts simplicially. Let $\ve=(\ve_{ij})_{i,j \in I}$ be an exchange matrix and set $n:=|I \setminus I_0|$. A labeled $n$-regular tree $T$ is an (infinite) $n$-valent tree with a labeling $\{\text{edges of $T$}\} \to \{1,\dots,n\}$ which induces a bijection $\{ \text{edges incident to $v$} \} \xrightarrow{\sim} \{1,\dots,n\}$ for each vertex $v$ of $T$.
Choose a vertex $v_0$ of $T$ and assign a matrix to each vertex of $T$ by the following rules:
\begin{itemize}\itemsep=0pt
\item
The matrix assigned to $v_0$ is the initial exchange matrix $\ve$.
\item
The matrices assigned to a pair of vertices connected by an edge labeled by $k$ are related by the matrix mutation $\mu_k$.
\end{itemize}
Let $T_{|\ve|}$ denotes the $n$-regular tree equipped with such an assignment of exchange matrices on the vertices. Let $D$ be the subgroup of $\operatorname{Aut}(T_{|\ve|})$ which consists of elements which preserve the assigned matrices.

\begin{lem} We have a natural surjective homomorphism $\Phi_{v_0}\colon D \to \widehat\Gamma_{\ve}$.
\end{lem}

\begin{proof}
For an element $\gamma \in D$, the two matrices assigned to $v_0$ and $\gamma^{-1}(v_0)$ are the same. Let $p(\gamma)=e_{i_k}\dots e_{i_1}$ be the unique edge path in $T_{|\ve|}$ from $v_0$ to $\gamma^{-1}(v_0)$, where an edge $e_i$ has the labeling $i$. Then there exists a seed isomorphism $\sigma=\sigma(\gamma)$ from $\ve$ to $\mu_{i_k}\dots\mu_{i_1}(\ve)$ such that $\gamma(e_i)=e_{\sigma(i)}$ for all $i \in I$. For another element $\gamma' \in D$, let $p(\gamma')=e_{j_l}\dots e_{j_1}$ be the edge path from $v_0$ to $\gamma'^{-1}(v_0)$.
Then $\gamma^{-1}(p(\gamma'))=e_{\sigma^{-1}(j_l)}\dots e_{\sigma^{-1}(j_1)}$ is the edge path from $\gamma^{-1}(v_0)$ to $(\gamma'\gamma)^{-1}(v_0)$.
Hence we have
\begin{align*}
\Phi_{v_0}(\gamma')\Phi_{v_0}(\gamma)
	&=\sigma(\gamma')\mu_{j_l}\cdots \mu_{j_1}\sigma(\gamma)\mu_{i_k}\cdots \mu_{i_1}
	 =\sigma(\gamma')\sigma(\gamma)\mu_{\sigma^{-1}(j_l)}\cdots \mu_{\sigma^{-1}(j_1)}\mu_{i_k}\cdots \mu_{i_1} \\
	&=\sigma(\gamma'\gamma)\mu_{\sigma^{-1}(j_l)}\cdots \mu_{\sigma^{-1}(j_1)}\mu_{i_k}\cdots \mu_{i_1}
	 =\Phi_{v_0}(\gamma'\gamma).
\end{align*}

Thus the map $\Phi_{v_0}\colon D \to \widehat\Gamma_{\ve}$ given by $\gamma \mapsto \seq$ is a group homomorphism. It is clearly surjective.
\end{proof}

The graph $\wE_{|\ve|}:=T_{|\ve|}/\operatorname{ker}\Phi_{v_0}$ is called the \emph{saturated exchange graph}. The saturated cluster modular group acts on $\wE_{|\ve|}$ as graph automorphisms via $\Phi_{v_0}$. We write $\widehat\Gamma_{v_0}:=\Phi_{v_0}^{-1}\big(\widehat\Gamma_{\ve}\big) \subset \operatorname{Aut}\big(\wE_{|\ve|}\big)$.
\begin{rem}\label{rem:satexgraph}\quad
\begin{enumerate}\itemsep=0pt
\item
An element $\gamma$ of $D$ belongs to the subgroup $\operatorname{ker}\Phi_{v_0}$ if and only if there exists a vertex $v$ equipped with $\ve$ and the edge path from $v$ to $\gamma(v)$ is given by a concatenation of standard sequences.
\item
If we change the basepoint of $T$ to a vertex $v_1$ which is connected with $v_0$ by an edge $e_k$, then we have $\Phi_{v_1} = \operatorname{Ad}_{\mu_k} \circ \Phi_{v_0}\colon D \to \widehat\Gamma_{\mu_k(\ve)}$. Here $\operatorname{Ad}_{\mu_k}$ denotes the conjugation by $\mu_k$.
\end{enumerate}
\end{rem}

The \emph{exchange graph} of the mutation class $|\bi|$ is defined as follows \cite{FZ07}. Starting with a $n$-regular tree $T$ with a fixed vertex $v_0$, we assign a seed to each vertex of $T$ by the following rules:
\begin{itemize}\itemsep=0pt
\item
The seed assigned to $v_0$ is the initial seed $\bi$.
\item
The seeds assigned to a pair of vertices connected by an edge labeled by $k$ are related by the seed mutation $\mu_k$.
\end{itemize}
We identify two vertices with the same seeds, and the resulting graph $\mathbb{E}_{|\bi|}$ is the exchange graph. From Lemma~\ref{standard relations} and Remark~\ref{rem:satexgraph} we have a natural covering map $\wE_{|\ve|} \to \mathbb{E}_{|\bi|}$.
\begin{lem}
The group homomorphism $\widehat\Gamma_{|\ve|} \to \Gamma_{|\bi|}$ is an isomorphism if and only if the covering map $\wE_{|\ve|} \to \mathbb{E}_{|\bi|}$ is a homeomorphism.
\end{lem}

\begin{proof}
Both conditions are equivalent to the condition that relations among mutations are generated by standard ones.
\end{proof}

\begin{dfn}[saturated modular complex]
To each cycle in $\wE_{|\ve|}$ which is given by a standard sequence, we attach a 2-cell given by the corresponding standard polygon. The resulting 2-complex $\widehat\M_{|\ve|}$ is called the \emph{saturated modular complex}.
\end{dfn}
Here is a simple but important lemma:
\begin{lem}
The saturated modular complex $\widehat\M_{|\ve|}$ is simply-connected.
\end{lem}

\begin{proof}From Remark~\ref{rem:satexgraph}, each combinatorial cycle $\alpha$ in $\widehat\M_{|\ve|}$ is a concatenation of standard cycles. Since each standard cycle spans a 2-cell given by the corresponding standard polygon, $\alpha$ can be contracted to a point along these 2-cells. Therefore $\widehat\M_{|\ve|}$ is simply-connected.
\end{proof}

\subsection{Seeds of finite mutation type}
An exchange matrix $\ve$ is said to be of \emph{finite mutation type} if the mutation class $|\ve|$ is a finite set. If $\ve$ is mutation-equivalent to a Dynkin quiver, then it is said to be of \emph{finite type}. Otherwise it is of \emph{infinite type}.

\begin{thm}[Fomin--Zelevinsky \cite{FZ03}] For a seed $\bi=(\ve,(A_i)_{i \in I}, (X_i)_{i \in I})$, the underlying exchange matrix $\ve$ is of finite type if and only if the mutation class $|\bi|$ is a finite set.
\end{thm}

Indeed, the latter condition is the original definition of a seed of finite type. An important point is that whether a seed is of finite type is determined by its exchange matrix.
The following theorems give the classification of exchange matrices of finite mutation type.

\begin{thm}[Felikson--Shapiro--Tumarkin~\cite{FeST12}]\label{thm:FSTskewsym}
Any skew-symmetric exchange matrix with size $\geq 3$ of finite mutation type is either obtained by an ideal triangulation of a marked surface or contained in one of the eleven mutation classes: $E_6$, $E_7$, $E_8$, $\tilde{E}_6$, $\tilde{E}_7$, $\tilde{E}_8$, $E_6^{(1,1)}$, $E_7^{(1,1)}$, $E_8^{(1,1)}$, $X_6$, $X_7$.
\end{thm}
The quivers representing the eleven mutation classes are given in Fig.~6.1 of~\cite{FeST12}. The first three are of finite type.

\begin{thm}[Felikson--Shapiro--Tumarkin~\cite{FeST}]\label{thm:FSTgeneral}
An exchange matrix with size $\geq 3$ which is not skew-symmetric, is of finite mutation type if and only if it is s-decomposable or contained in one of the seven mutation classes: $\tilde{G}_2$, $F_4$, $\tilde{F}_4$, $G_2^{(*,+)}$, $G_2^{(*,*)}$, $F_4^{(*,+)}$, $F_4^{(*,*)}$
\end{thm}
The quivers representing the seven mutation classes are given in Fig.~1.1 of~\cite{FeST}. Only $F_4$ is of finite type. Here an exchange matrix is said to be \emph{s-decomposable} if it is obtained by gluing certain ``blocks'', see~\cite{FeST}.
We say that a seed $\bi=(\ve,(A_i)_{i \in I}, (X_i)_{i \in I})$ is of finite mutation type if the underlying exchange matrix is.
\begin{lem}If a seed $\bi$ is of finite mutation type and infinite type, then the cluster modular group $\Gamma_{|\bi|}$ is a finitely generated infinite group.
\end{lem}

\begin{proof}In this case we have infinitely many seeds whose underlying exchange matrix coincides with the initial one~$\ve$. Each such seed gives rise to a distinct element of the cluster modular group, hence the cluster modular group is infinite. The cluster modular group is finitely generated by Lemma~\ref{lem:finitegraph}.
\end{proof}

\section{Brown's algorithm for a group acting on a CW complex}\label{section: Brown}
Let us recall Brown's algorithm~\cite{Brown84} with a suitable specialization for our purpose. Let~$G$ be a~group acting on a simply-connected CW complex $X$ preserving the cell structure. It is known that if the action has no fixed points, then $G \cong \pi_1(X)$. It is not true under existence of a fixed point (which corresponds to existence of a quiver isomorphism in our cluster setting). Brown's algorithm enables us to compute a presentation of~$G$ even in this situation.

\subsection{General setting}
We prepare some notations and terminology. The set of $k$-cells of a CW complex $X$ is denoted by $C_k(X)$. A 1-cell $\sigma \in C_1(X)$ is said to be \emph{inverted} if there exists an element $g \in G$ which fixes~$\sigma$ and reverses the orientation of~$\sigma$. If $\sigma$ is not a loop in~$X$, which is the case we will deal with, the condition means that $g$ transposes two endpoints of~$\sigma$. In the following, we will assume the following condition:
\begin{equation}\label{non-inverted}
\text{each 1-cell of $X$ has distinct endpoints and is not inverted.}
\end{equation}
A subtree $T$ of $X$ is called a \emph{tree of representatives} for $X$ mod $G$ if the vertices of $T$ form a set of representatives for the vertices of~$X$ mod~$G$.
Such a tree always exists and the 1-cells of~$T$ are inequivalent mod $G$. By an \emph{edge} of $X$ we mean an oriented 1-cell of~$X$. Let us denote the set of edges of~$X$ by~$E(X)$, which is again a $G$-set.

Let $\pi\colon E(X) \to C_1(X)$ be the projection forgetting the orientation of each 1-cell. A $G$-equivariant section $P\colon C_1(X) \to E(X)$ is called an \emph{orientation of the CW complex $X$}. Such an orientation exists under the assumption~(\ref{non-inverted}). The initial (resp. terminal) vertex of an edge $e \in E(X)$ is denoted by $o(e)$ (resp.~$t(e)$). We represent an edge~$e$ by the symbol $\big(o(e) \xrightarrow{e} t(e)\big)$.

For a 2-cell in $X$, the attaching map $\partial \Delta^2 \to X^{(1)}$ is represented by a combinatorial cycle $\tau=e_1\cdots e_n$ such that $o(e_1)=t(e_n) \in C_0(T)$. Here $\Delta^2$ denotes the standard 2-simplex, and we read the combinatorial cycle from the left to the right by convention.
\begin{dfn}\label{Brown data}
A choice of the following data $(T,P,E^+,F)$ is called a \emph{data of representatives} for the action of $G$ on $X$:
\begin{enumerate}\itemsep=0pt
\item $T \subseteq X^{(1)}$ is a tree of representatives for $X$ mod $G$.
\item $P\colon C_1(X) \to E(X)$ is an orientation of $X$.
\item $E^+ \subseteq P(C_1(X))$ is a set of representatives mod $G$ such that $P(C_1(T)) \subseteq E^+$ and $o(e) \in C_0(T)$ for each edge $e$ in $E^+$.
\item $F \subseteq C_2(X)$ is a set of representatives mod $G$ together with a choice of a combinatorial cycle representing the attaching map for each 2-cell.
\end{enumerate}
\end{dfn}

Fixing a data of representatives, we can compute a presentation of the group $G$. For each edge $e \in E^+$, there is a unique vertex $w_e \in C_0(T)$ which is $G$-equivalent to the terminal vertex~$t(e)$. Choose an element $g_e \in G$ such that
\begin{equation}\label{Brown generator}
t(e)=g_e(w_e).
\end{equation}
Here we choose $g_e:=1 \in G$ if $e$ is contained in $T$ by convention. These elements together with the elements of the isotropy group $G_v:=\{g \in G\,|\,g(v)=v\}$ ($v \in C_0(T)$) form a generating set of $G$. We call them \emph{Brown generators}. We shall describe the relations among the Brown generators. Note that each edge $e \in E(X)$ such that $v:=o(e) \in C_0(T)$ has one of the following forms:
\begin{enumerate}\itemsep=0pt
\item[(a)] $e=\big(v \xrightarrow{he} hg_ew_e\big)$, where $e \in E^+$ and $h \in G_v$.
\item[(b)] $e=\big(w_e \xrightarrow{hg_e^{-1}\bar{e}} hg_e^{-1}v\big)$, where $e \in E^+$, $\bar{e}$ is the opposite edge and $h \in G_{w_e}$.
\end{enumerate}
Then we set \[
g:=\begin{cases} hg_e & \text{(case (a))}, \\
hg_e^{-1} & \text{(case (b))}. \end{cases}
\]
For a combinatorial cycle $\tau=e_1\cdots e_n$ such that $o(e_1)=t(e_n) \in C_0(T)$, the above rule provides us with elements $g_1^\tau,\dots,g_n^\tau \in G$ such that $t(e_i) \in g_1^\tau\cdots g_{i}^\tau(C_0(T))$. Let $g^\tau:=g_1^\tau\cdots g_n^\tau$. Note that $g^\tau \in G_{o(e_1)}$.

\begin{thm}[Brown~\cite{Brown84}]\label{thm Brown}
Let $G$ be a group acting on a simply-connected CW complex $X$ preserving the cell structure. Assume the condition~\eqref{non-inverted}. Then $G$ is generated by the elements~$g_e$ $(e \in E^+)$ and the isotropy groups~$G_v$ $(v \in C_0(T))$, subject to the following relations:
\begin{enumerate}\itemsep=0pt
\item[$1.$] $g_e=1$ if $e$ is contained in $T$.
\item[$2.$] $g_e^{-1}i_e(g)g_e=c_e(g)$ for any $e \in E^+$ and $g \in G_e$, where $i_e\colon G_e \to G_{o(e)}$ is the inclusion and $c_e\colon G_e \to G_{w_e}$ is defined by $g \mapsto g_e^{-1}gg_e$.
\item[$3.$] $g^\tau=g_1^\tau\cdots g_n^\tau$ for any $\tau \in F$.
\end{enumerate}
\end{thm}
We call the relations of type~(2) the \emph{isotropy relations}, and those of type (3) the \emph{face relations}.

\subsection{Application to the cluster modular groups}\label{apply cluster}
In this section, we collect basic properties of the action of the saturated cluster modular group on the saturated cluster modular complex which are relevant to the application of the above method.
Let $\ve=(\ve_{ij})_{i,j \in I}$ be an exchange matrix. Let $V:=C_0\big(\wM_{|\ve|}\big)$ be the set of vertices of the saturated modular complex. First note that we have a natural surjective map $\operatorname{Mat}\colon V \to |\ve|$ which extracts the exchange matrix assigned to each vertex.

\begin{lem}\label{vertex class}
The map $\operatorname{Mat}$ induces a bijection $\operatorname{Mat}\colon V/\wG_{|\ve|} \to |\ve|$.
\end{lem}

\begin{proof}The surjective map $\operatorname{Mat}\colon V \to |\ve|$ descends to a map $\operatorname{Mat}\colon V/\wG_{|\ve|} \to |\ve|$, since the action of the saturated cluster modular group preserves the assigned exchange matrices. If two vertices~$v_0$ and~$v_1$ are assigned the same matrices, then there exists an element $\gamma \in \operatorname{Aut}(D)$ such that $\gamma(v_0)=v_1$. Hence $v_0$ is sent to $v_1$ by an element of $\wG_{|\ve|}$. Thus the map $\operatorname{Mat}\colon V/\wG_{|\ve|} \to |\ve|$ is bijective.
\end{proof}

\begin{lem}\label{without inversion} The action of the saturated cluster modular group on the saturated modular complex $\widehat\M_{|\ve|}$ satisfies the condition~\eqref{non-inverted}.
\end{lem}

\begin{proof}
First note that there are no loops in the saturated cluster modular complex, since we have no standard relation of length $1$. Assume that an element $\phi \in \wG_{|\ve|}$, $\phi\neq 1$, fixes an edge $e$ of $\widehat\M_{|\ve|}$. Let $x$ and $y$ be endpoints and $i \in I$ the label of $e$. Then $\phi$ is of the form $\phi=\sigma\mu_i$ as an element of $\wG_{[x]}$, where $\sigma$ is an involutive quiver isomorphism. Now suppose we have $\phi(x)=y$ and $\phi(y)=x$. Then we have $\phi^2 \in \operatorname{Aut}([x])$, namely, $\phi^2$ must be given by a seed isomorphism~$\tau$ as an element of $\wG_{[x]}$. Hence we have a non-trivial relation $\tau=\mu_{\sigma^{-1}(i)}\mu_i$, which cannot be written as a composition of the standard relations. Thus we have a contradiction.
\end{proof}

If an edge $e$ of the saturated modular complex has endpoints $v_0$ and $v_1$, the above lemma provides an inclusion of isotropy groups $\big(\widehat\Gamma_{|\ve|}\big)_e \subseteq \big(\widehat\Gamma_{|\ve|}\big)_{v_0} \cap \big(\widehat\Gamma_{|\ve|}\big)_{v_1}$. An element $\sigma$ of $(\widehat\Gamma_{|\ve|})_e$ is called a \emph{simultaneous seed isomorphism} of the pair $(v_0,v_1)$. We also say that~$\sigma$ is a simultaneous seed isomorphism of the pair of underlying exchange matrices $([v_0],[v_1])$.

\begin{lem}\label{edge class}
The $\wG_{|\ve|}$-orbit of an oriented edge in $\wM_{|\ve|}$ is determined by a triple $(\ve_1,\ve_2;k)$ of two exchange matrices $\ve_1, \ve_2 \in |\ve|$ and an index $k \in I$ such that $\ve_2=\mu_k(\ve_1)$, modulo the equivalence relation generated by $(\ve_1,\ve_2;k) \sim (\sigma.\ve_1,\sigma.\ve_2;\sigma(k))$ for seed permutations $\sigma$.
\end{lem}

\begin{proof}Fix a tree $T$ of representative. Then each oriented edge is translated by an element of~$\wG_{|\ve|}$ so that its origin belongs to~$C_0(T)$. If we have two oriented edges with the same origin, they are translated each other by an element of~$\wG_{|\ve|}$ if and only if the corresponding triples are equivalent.
\end{proof}

\begin{lem}\label{face class}The $\wG_{|\ve|}$-orbit of a standard $(h+2)$-gon cycle $C$ in $\wM_{|\ve|}$ is determined by a triple $(\ve';k,l)$ of an exchage matrix $\ve' \in |\ve|$ and indices $k,l \in I-I_0$ such that $\ve'_{kl}=-p\ve'_{lk}=p$, modulo the equivalence relation generated by $(\ve';k,l) \sim (\sigma.\ve';\sigma(k),\sigma(l))$. Here $(p,h)=(0,2)$, $(1,3)$, $(2,4)$ or $(3,6)$.
\end{lem}

\begin{proof}
The base point of each standard cycle is translated by an element of $\wG_{|\ve|}$ to a point belonging to a fixed tree of representative. Then the assertion is clear.
\end{proof}

A standard $(h+2)$-gon cycle which path through a vertex $v \in V$ and determined by a pair of indices~$k$, $l$ is denoted by $C_{h+2}(k,l)_v$. A cycle may have several expressions.
\begin{dfn}[the saturated modular graph]
We define the \emph{saturated modular graph} $\widehat{G}_{|\ve|}$ to be the graph whose vertices are exchange matrices in $|\ve|$ and edges are $\wG_{|\ve|}$-orbits of edges in~$\wM_{|\ve|}$.
\end{dfn}

From Lemma~\ref{face class}, each $\wG_{|\ve|}$-orbit of a standard cycle is represented by a circuit in the saturated modular graph~$\widehat{G}_{|\ve|}$. The following lemma is clear from the definition.

\begin{lem}\label{lem:finitegraph}
If $\ve$ is of finite mutation type, then the saturated modular graph is a finite graph.
\end{lem}

\section[Presentations of $\widehat\Gamma_{X_7}$ and $\widehat\Gamma_{X_6}$]{Presentations of $\boldsymbol{\widehat\Gamma_{X_7}}$ and $\boldsymbol{\widehat\Gamma_{X_6}}$}\label{section: proof}

In this section, based on the method established in Section~\ref{section: Brown}, we give finite presentations of the saturated cluster modular groups of type~$X_7$ and~$X_6$.

\subsection[The saturated cluster modular group of type $X_7$]{The saturated cluster modular group of type $\boldsymbol{X_7}$}

\begin{figure}[t]\centering
\scalebox{0.9}{
\begin{tikzpicture}
\begin{scope}[>=latex]
\fill (0,0) circle(2pt) coordinate(A) node[above]{$0$};
\fill (A) ++(165: 2) circle(2pt) coordinate(B) node[left]{$1$};
\fill (A) ++(135: 2) circle(2pt) coordinate(C) node[left]{$2$};
\fill (A) ++(45: 2) circle(2pt) coordinate(D) node[right]{$3$};
\fill (A) ++(15: 2) circle(2pt) coordinate(E) node[right]{$4$};
\fill (A) ++(285: 2) circle(2pt) coordinate(F) node[below]{$5$};
\fill (A) ++(255: 2) circle(2pt) coordinate(G) node[below]{$6$};

\draw[->,shorten >=2pt,shorten <=2pt] (A) -- (B) [thick];
\draw[->,double,shorten >=2pt,shorten <=2pt] (B) -- (C) [thick];
\draw[->,shorten >=2pt,shorten <=2pt] (C) -- (A) [thick];
\draw[->,shorten >=2pt,shorten <=2pt] (A) -- (D) [thick];
\draw[->,double,shorten >=2pt,shorten <=2pt] (D) -- (E) [thick];
\draw[->,shorten >=2pt,shorten <=2pt] (E) -- (A) [thick];
\draw[->,shorten >=2pt,shorten <=2pt] (A) -- (F) [thick];
\draw[->,double,shorten >=2pt,shorten <=2pt] (F) -- (G) [thick];
\draw[->,shorten >=2pt,shorten <=2pt] (G) -- (A) [thick];

\draw (A) ++(315: 3) node{$Q_0$};

\fill (6,0) circle(2pt) coordinate(A1) node[above]{$0$};
\fill (A1) ++(165: 2) circle(2pt) coordinate(B1) node[left]{$1$};
\fill (A1) ++(135: 2) circle(2pt) coordinate(C1) node[left]{$2$};
\fill (A1) ++(45: 2) circle(2pt) coordinate(D1) node[right]{$3$};
\fill (A1) ++(15: 2) circle(2pt) coordinate(E1) node[right]{$4$};
\fill (A1) ++(285: 2) circle(2pt) coordinate(F1) node[below]{$5$};
\fill (A1) ++(255: 2) circle(2pt) coordinate(G1) node[below]{$6$};

\draw[->,shorten >=2pt,shorten <=2pt] (B1) -- (A1) [thick];
\draw[->,shorten >=2pt,shorten <=2pt] (A1) -- (C1) [thick];
\draw[->,shorten >=2pt,shorten <=2pt] (D1) -- (A1) [thick];
\draw[->,shorten >=2pt,shorten <=2pt] (A1) -- (E1) [thick];
\draw[->,shorten >=2pt,shorten <=2pt] (F1) -- (A1) [thick];
\draw[->,shorten >=2pt,shorten <=2pt] (A1) -- (G1) [thick];

\draw[->,shorten >=2pt,shorten <=2pt] (B1) -- (C1) [thick];
\draw[->,shorten >=2pt,shorten <=2pt] (C1) -- (D1) [thick];
\draw[->,shorten >=2pt,shorten <=2pt] (D1) -- (E1) [thick];
\draw[->,shorten >=2pt,shorten <=2pt] (E1) -- (F1) [thick];
\draw[->,shorten >=2pt,shorten <=2pt] (F1) -- (G1) [thick];
\draw[->,shorten >=2pt,shorten <=2pt] (G1) -- (B1) [thick];

\draw[->,shorten >=2pt,shorten <=2pt] (E1) -- (B1) [thick];
\draw[->,shorten >=2pt,shorten <=2pt] (C1) -- (F1) [thick];
\draw[->,shorten >=2pt,shorten <=2pt] (G1) -- (D1) [thick];

\draw (A1) ++(315: 3) node{$Q_1$};

\end{scope}
\end{tikzpicture}}
\caption{Two quivers in $|X_7|$.}\label{fig: X_7}
\end{figure}

The mutation class $X_7$ consists of two quivers~$Q_0$ and~$Q_1$. They are shown in Fig.~\ref{fig: X_7}. Let $v_0 \in V$ be a vertex of the saturated modular complex such that $\operatorname{Mat}(v_0)=Q_0$. Let us fix a data of representatives as follows:
\begin{enumerate}\itemsep=0pt
\item The tree $T$ consists of two vertices $v_0$ and $v_1:=\mu_0(v_0)$ together with an edge $e_0:=\big(v_0\xrightarrow{\mu_0}v_1\big)$.
\item The set $E^+$ of representative of oriented edges consists of three oriented edges $e_0$, $e_1:=\big(v_0\xrightarrow{\mu_1}\mu_1(v_0)\big)$ and $e_2:=\big(v_1\xrightarrow{\mu_1}\mu_1(v_1)\big)$. There is an orientation of $\wM_{X_7}$ which extends them.
\item The set $F:=\{\tau_1,\tau_2,\tau_3,\tau_4\}$ of representative of 2-cells consists of two square cycles $\tau_1:=C_4(1,4)_{v_0}$, $\tau_2:=C_4(1,3)_{v_1}$ and two pentagon cycles $\tau_3:=C_5(0,1)_{v_0}$, $\tau_4:=C_5(1,4)_{v_1}$.
\end{enumerate}

\begin{figure}[t]\centering
\begin{tikzpicture}
\fill (0,0) circle(2pt) coordinate(A) node[above]{$Q_0$};
\fill (2,0) circle(2pt) coordinate(B) node[above]{$Q_1$};
\draw[thick] (A)--(B) node[midway,above]{$[e_0]$};
\draw[dashed] (A) to[out=145, in=90] node[above]{$[e_1]$} (-1.5,0) to[out=270, in=225] (A);
\draw[dashed] (B) to[out=45, in=90] node[above]{$[e_2]$} (3.5,0) to[out=270, in=315] (B);
\end{tikzpicture}
\caption{The modular graph $\widehat{G}_{X_7}$. The tree $T$ is shown by a thick line. The set $E^+$ is shown by dashed lines.}\label{fig: graphX_7}
\end{figure}

The tree $T$ and the set $E^+$ are shown in Fig.~\ref{fig: graphX_7}. Then the following proposition can be easily verified using Lemmas~\ref{vertex class}--\ref{face class}.
\begin{prop}
The data $(T,P,E^+,F)$ determines a data of representatives for the action of~$\wG_{X_7}$ on~$\wM_{X_7}$.
\end{prop}
Then we choose the Brown generators as follows:
\begin{enumerate}\itemsep=0pt
\item
The isotropy group of the vertex $v_0$ is the image of the group homomorphism
\[
\iota\colon \ \mathfrak{S}(\{1,3,5\}) \to \mathfrak{S}(\{0,1,2,3,4,5,6\})
\]
given by $\iota(\sigma)\colon 0 \mapsto 0$, $k \mapsto \sigma(k)$, $k+1 \mapsto \sigma(k)+1$ for $k \in \{1,3,5\}$. Here $\mathfrak{S}(X)$ denotes the symmetric group of a finite set $X$. We write a cyclic permutation $k_1 \mapsto \dots \mapsto k_n \mapsto k_1$ of $\{1,3,5\}$ as $\sigma_{k_1\dots k_n}$. It turns out that the isotropy group of $v_1$ is the same, hence $\mathcal{S}_3:=\iota(\mathfrak{S}(\{1,3,5\}))$ consists of simultaneous seed isomorphisms.
\item
Set $g_{e_0}:=1$, $g_{e_1}:=(1\ 2)\mu_1 \in \wG_{v_0}$ and $g_{e_2}:=(0\ 1\ 2)(3\ 4\ 5\ 6)\mu_1 \in \wG_{v_1}$. Then they satisfy the condition~(\ref{Brown generator}). Changing their basepoints to the vertex $v_0$ using the edge $e_0$ (see Remark~\ref{rem:satexgraph}), we get $\phi_1:=(1\ 2)\mu_1$ and $\psi_1:=(0\ 1\ 2)(3\ 4\ 5\ 6)\mu_2\mu_1\mu_0$, which we regard as elements in~$\wG_{Q_0}$.
\end{enumerate}

From Theorem~\ref{thm Brown}, the elements in $\mathcal{S}_3$ together with $\phi_1$ and $\psi_1$ generate the saturated cluster modular group. Let us investigate the relations among them. To simplify the computations, let us introduce auxiliary elements $\phi_3:=\sigma_{135}\phi_1\sigma_{135}^{-1}$ and $\phi_5:=\sigma_{153}\phi_1\sigma_{153}^{-1}$. Similarly define~$\psi_3$ and~$\psi_5$. Then the relations are determined as follows.

{\bf Isotropy relations.}
The isotropy group of the edge $e_1$ is generated by $\sigma_{35}$. Then the isotropy relation implies $\phi_1^{-1}\sigma_{35}\phi_1=\sigma_{35}$. Similarly, the isotropy group of the edge $e_2$ is generated by~$\sigma_{35}$. Then the isotropy relation implies $\psi_1^{-1}\sigma_{35}\psi_1=\sigma_{35}$.

{\bf Face relations.} From the cycle $\tau_1$ we get
\[
\phi_3\phi_1^{-1}\phi_3^{-1}\phi_1=(3\ 4)\mu_3(1\ 2)\mu_2(3\ 4)\mu_4(1\ 2)\mu_1=\mu_4\mu_1\mu_4\mu_1=1.
\]
From the cycle $\tau_3$ we get
\begin{align*}
\sigma_{35}\phi_1^{-1}\psi_1^2&=(3\ 5)(4\ 6)(1\ 2)\mu_2(0\ 1\ 2)(3\ 4\ 5\ 6)\mu_2\mu_1\mu_0(0\ 1\ 2)(3\ 4\ 5\ 6)\mu_2\mu_1\mu_0 \\
&=(0\ 1)\mu_0\mu_1\mu_0\mu_1\mu_0=1.
\end{align*}

In order to investigate $\tau_2$ and $\tau_4$, we represent each element in $\wG_{v_1}$: $\psi_1=(0\ 1\ 2)(3\ 4\ 5\ 6)\mu_1$, $\psi_3=(0\ 3\ 4)(5\ 6\ 1\ 2)\mu_3$ and $\psi_5=(0\ 5\ 6)(1\ 2\ 3\ 4)\mu_5$.
Then from the cycle $\tau_2$ we get
\begin{align*}
\psi_3^{-1}\psi_1\psi_3^{-1}\psi_1&
 =(0\ 4\ 3)(6\ 5\ 2\ 1)\mu_4(0\ 1\ 2)(3\ 4\ 5\ 6)\mu_1(0\ 4\ 3)(6\ 5\ 2\ 1)\mu_4(0\ 1\ 2)(3\ 4\ 5\ 6)\mu_1\\
&=\mu_3\mu_1\mu_3\mu_1=1.
\end{align*}

From the cycle $\tau_4$ we get
\begin{align*}
\sigma_{153}\psi_5\psi_1\psi_3\psi_5\psi_1
&=(1\ 5\ 3)(2\ 6\ 4)(0\ 5\ 6)(1\ 2\ 3\ 4)\mu_5(0\ 1\ 2)(3\ 4\ 5\ 6)\mu_1 \\
&\qquad\ (0\ 3\ 4)(5\ 6\ 1\ 2)\mu_3(0\ 5\ 6)(1\ 2\ 3\ 4)\mu_5(0\ 1\ 2)(3\ 4\ 5\ 6)\mu_1\\
&=(1\ 4)\mu_1\mu_4\mu_1\mu_4\mu_1=1.
\end{align*}
Summarizing these relations, we obtain Theorem~\ref{introthm: X_7}.

\subsection[The saturated cluster modular group of type $X_6$]{The saturated cluster modular group of type $\boldsymbol{X_6}$}

\begin{figure}[t]\centering
\scalebox{0.9}{
\begin{tikzpicture}
\begin{scope}[>=latex]
\fill (-3,0) circle(2pt) coordinate(O) node[above]{$1$};
\fill (O) ++(165: 2) circle(2pt) coordinate(A) node[left]{$2$};
\fill (O) ++(135: 2) circle(2pt) coordinate(B) node[left]{$3$};
\fill (O) ++(45: 2) circle(2pt) coordinate(C) node[right]{$4$};
\fill (O) ++(15: 2) circle(2pt) coordinate(D) node[right]{$5$};
\fill (O) ++(270: 2) circle(2pt) coordinate(E) node[below]{$6$};

\draw[->,shorten >=2pt,shorten <=2pt] (O) -- (A) [thick];
\draw[->,double,shorten >=2pt,shorten <=2pt] (A) -- (B) [thick];
\draw[->,shorten >=2pt,shorten <=2pt] (B) -- (O) [thick];
\draw[->,shorten >=2pt,shorten <=2pt] (O) -- (C) [thick];
\draw[->,double,shorten >=2pt,shorten <=2pt] (C) -- (D) [thick];
\draw[->,shorten >=2pt,shorten <=2pt] (D) -- (O) [thick];
\draw[->,shorten >=2pt,shorten <=2pt] (O) -- (E) [thick];

\draw (O) ++(315: 3) node{$Q_0$};

\fill (3,0) circle(2pt) coordinate(O1) node[above]{$1$};
\fill (O1) ++(165: 2) circle(2pt) coordinate(A1) node[left]{$2$};
\fill (O1) ++(135: 2) circle(2pt) coordinate(B1) node[left]{$3$};
\fill (O1) ++(45: 2) circle(2pt) coordinate(C1) node[right]{$4$};
\fill (O1) ++(15: 2) circle(2pt) coordinate(D1) node[right]{$5$};
\fill (O1) ++(270: 2) circle(2pt) coordinate(E1) node[below]{$6$};

\draw[->,shorten >=2pt,shorten <=2pt] (O1) -- (A1) [thick];
\draw[->,shorten >=2pt,shorten <=2pt] (B1) -- (O1) [thick];
\draw[->,shorten >=2pt,shorten <=2pt] (O1) -- (D1) [thick];
\draw[->,shorten >=2pt,shorten <=2pt] (C1) -- (O1) [thick];
\draw[->,shorten >=2pt,shorten <=2pt] (E1) -- (O1) [thick];

\draw[->,shorten >=2pt,shorten <=2pt] (A1) -- (C1) [thick];
\draw[->,shorten >=2pt,shorten <=2pt] (D1) -- (B1) [thick];
\draw[->,shorten >=2pt,shorten <=2pt] (C1) -- (D1) [thick];
\draw[->,shorten >=2pt,shorten <=2pt] (B1) -- (A1) [thick];
\draw[->,shorten >=2pt,shorten <=2pt] (A1) -- (E1) [thick];
\draw[->,shorten >=2pt,shorten <=2pt] (D1) -- (E1) [thick];

\draw (O1) ++(315: 3) node{$Q_1$};


\draw (0,-5) coordinate(O2);
\fill (O2) ++(0: 2) circle(2pt) coordinate(A2) node[right]{$1$};
\fill (O2) ++(60: 2) circle(2pt) coordinate(B2) node[above]{$6$};
\fill (O2) ++(120: 2) circle(2pt) coordinate(C2) node[above]{$2$};
\fill (O2) ++(180: 2) circle(2pt) coordinate(D2) node[left]{$4$};
\fill (O2) ++(240: 2) circle(2pt) coordinate(E2) node[below]{$3$};
\fill (O2) ++(300: 2) circle(2pt) coordinate(F2) node[below]{$5$};

\draw[->,shorten >=2pt,shorten <=2pt] (A2) -- (B2) [thick];
\draw[->,shorten >=2pt,shorten <=2pt] (B2) -- (C2) [thick];
\draw[->,shorten >=2pt,shorten <=2pt] (C2) -- (D2) [thick];
\draw[->,shorten >=2pt,shorten <=2pt] (D2) -- (E2) [thick];
\draw[->,shorten >=2pt,shorten <=2pt] (E2) -- (F2) [thick];
\draw[->,shorten >=2pt,shorten <=2pt] (F2) -- (A2) [thick];

\draw[->,shorten >=2pt,shorten <=2pt] (B2) -- (E2) [thick];
\draw[->,shorten >=2pt,shorten <=2pt] (D2) -- (A2) [thick];
\draw[->,shorten >=2pt,shorten <=2pt] (F2) -- (C2) [thick];

\draw (O2) ++(315: 3) node{$Q_2$};

\fill (-3,-10) circle(2pt) coordinate(O3) node[above]{$1$};
\fill (O3) ++(165: 2) circle(2pt) coordinate(A3) node[left]{$2$};
\fill (O3) ++(135: 2) circle(2pt) coordinate(B3) node[left]{$3$};
\fill (O3) ++(45: 2) circle(2pt) coordinate(C3) node[right]{$4$};
\fill (O3) ++(15: 2) circle(2pt) coordinate(D3) node[right]{$5$};
\fill (O3) ++(270: 2) circle(2pt) coordinate(E3) node[below]{$6$};

\draw[->,shorten >=2pt,shorten <=2pt] (O3) -- (B3) [thick];
\draw[->,shorten >=2pt,shorten <=2pt] (A3) -- (O3) [thick];
\draw[->,shorten >=2pt,shorten <=2pt] (O3) -- (C3) [thick];
\draw[->,shorten >=2pt,shorten <=2pt] (D3) -- (O3) [thick];
\draw[->,shorten >=2pt,shorten <=2pt] (O3) -- (E3) [thick];

\draw[->,shorten >=2pt,shorten <=2pt] (A3) -- (B3) [thick];
\draw[->,shorten >=2pt,shorten <=2pt] (D3) -- (C3) [thick];
\draw[->,shorten >=2pt,shorten <=2pt] (C3) -- (A3) [thick];
\draw[->,shorten >=2pt,shorten <=2pt] (B3) -- (D3) [thick];
\draw[->,shorten >=2pt,shorten <=2pt] (E3) -- (A3) [thick];
\draw[->,shorten >=2pt,shorten <=2pt] (E3) -- (D3) [thick];

\draw (O3) ++(315: 3) node{$Q_3$};

\fill (3,-10) circle(2pt) coordinate(O4) node[above]{$1$};
\fill (O4) ++(165: 2) circle(2pt) coordinate(A4) node[left]{$2$};
\fill (O4) ++(135: 2) circle(2pt) coordinate(B4) node[left]{$3$};
\fill (O4) ++(45: 2) circle(2pt) coordinate(C4) node[right]{$4$};
\fill (O4) ++(15: 2) circle(2pt) coordinate(D4) node[right]{$5$};
\fill (O4) ++(270: 2) circle(2pt) coordinate(E4) node[below]{$6$};

\draw[->,shorten >=2pt,shorten <=2pt] (O4) -- (A4) [thick];
\draw[->,double,shorten >=2pt,shorten <=2pt] (A4) -- (B4) [thick];
\draw[->,shorten >=2pt,shorten <=2pt] (B4) -- (O4) [thick];
\draw[->,shorten >=2pt,shorten <=2pt] (O4) -- (C4) [thick];
\draw[->,double,shorten >=2pt,shorten <=2pt] (C4) -- (D4) [thick];
\draw[->,shorten >=2pt,shorten <=2pt] (D4) -- (O4) [thick];
\draw[->,shorten >=2pt,shorten <=2pt] (E4) -- (O4) [thick];

\draw (O4) ++(315: 3) node{$Q_4$};

\end{scope}
\end{tikzpicture}}
\caption{Five quivers in $|X_6|$.}\label{fig: X_6}
\end{figure}

Next we investigate the mutation class $X_6$. Although the computation is a little more complicated in this case, the strategy is the same as before. The mutation class $X_6$ consists of five quivers $Q_0$,\dots, $Q_4$. They are shown in Fig.~\ref{fig: X_6}. Let $v_0 \in V$ be a vertex of the saturated modular complex such that $\operatorname{Mat}(v_0)=Q_0$. Let us fix a data of representatives as follows, see Fig.~\ref{fig: graphX_6}.
\begin{enumerate}\itemsep=0pt
\item The tree $T$ consists of five vertices $v_0$, $v_1:=(2\ 3)\mu_1(v_0)$, $v_2:=(3\ 5)\mu_6(v_1)$, $v_3:=(2\ 4)(1\ 6)\mu_6(v_2)$ and $v_4:=(4\ 5)\mu_1(v_3)$ together with four edges
\begin{alignat*}{3}
& e^0:=\big(v_0\xrightarrow{\mu_1}v_1\big), \qquad&&
e^1:=\big(v_1\xrightarrow{\mu_6}v_2\big), &\\
& e^2:=\big(v_2\xrightarrow{\mu_6}v_3\big),\qquad&&
e^3:=\big(v_3\xrightarrow{\mu_1}v_4\big).&
\end{alignat*}
Here the permutations are inserted in order to make the labelings of quiver vertices at $v_0,\dots,v_4$ consistent with those shown in Fig.~\ref{fig: X_6}.
\item The set $E^+:=\big\{e^0,\dots, e^3, e_0,\dots, e_6\big\}$ of representatives of oriented edges consists of eleven oriented edges, where
 \begin{alignat*}{4}
&e_0:=\big(v_0\xrightarrow{\mu_6}(v_4)\big), \qquad &&
e_1 :=\big(v_0\xrightarrow{\mu_2}\mu_2(v_0)\big),\qquad &&
e_2:=\big(v_4\xrightarrow{\mu_2}\mu_2(v_4)\big), &\\
& e_3:=\big(v_1\xrightarrow{\mu_2}\mu_2(v_1)\big), \qquad &&
e_4 :=\big(v_1\xrightarrow{\mu_3}\mu_3(v_1)\big), \qquad &&
e_5:=\big(v_2\xrightarrow{\mu_4}\mu_4(v_2)\big),& \\
& e_6:=\big(v_2\xrightarrow{\mu_3}\mu_3(v_2)\big).&&&& &
\end{alignat*}
There is an orientation of $\wM_{X_6}$ which extends them.
\item The set $F:=\{\tau_1,\dots,\tau_{11}\}$ of representatives of 2-cells consists of six square cycles \begin{alignat*}{4}
& \tau_1:=C_4(2,6)_{v_0}, \qquad &&
\tau_2 :=C_4(2,5)_{v_0}, \qquad &&
\tau_3:=C_4(2,5)_{v_4},&\\
& \tau_4:=C_4(3,4)_{v_1}, \qquad &&
\tau_5 :=C_4(2,5)_{v_0},\qquad &&
\tau_6:=C_4(4,6)_{v_1}&
\end{alignat*}
and five pentagon cycles \begin{alignat*}{4}
&\tau_7:=C_5(1,6)_{v_0}, \qquad &&
\tau_8:=C_5(1,2)_{v_0}, \qquad &&
\tau_9:=C_5(1,5)_{v_4}, &\\
& \tau_{10}:=C_5(2,4)_{v_1}, \qquad &&
\tau_{11} :=C_5(2,4)_{v_3}. &&&
\end{alignat*}
\end{enumerate}

\begin{figure}[t]\centering
\begin{tikzpicture}
\draw (0,0) coordinate(O);
\fill (O)++(126: 2) circle(2pt) coordinate(A) node[left]{$Q_0$};
\draw (A)++(0,1.5) coordinate(A');
\fill (O)++(198: 2) circle(2pt) coordinate(B) node[left]{$Q_1$};
\fill (O)++(270: 2) circle(2pt) coordinate(C) node[below]{$Q_2$};
\fill (O)++(342: 2) circle(2pt) coordinate(D) node[right]{$Q_3$};
\fill (O)++(54: 2) circle(2pt) coordinate(E) node[right]{$Q_4$};
\draw (E)++(0,1.5) coordinate(E');
\draw[thick] (A)--(B) node[midway,left]{$e^0$};
\draw[thick] (B)--(C) node[midway,below]{$e^1$};
\draw[thick] (C)--(D) node[midway,below]{$e^2$};
\draw[thick] (D)--(E) node[midway,right]{$e^3$};
\draw[dashed] (A) -- (E) node[midway,below]{$e_0$};
\draw[dashed] (A) to[out=135, in=180] node[left]{$e_1$} (A') to[out=0, in=45] (A);
\draw[dashed] (E) to[out=135, in=180] node[above]{$e_2$} (E') to[out=0, in=45] (E);
\draw[dashed] (B) to[out=325, in=215] node[above]{$e_3$} (D);
\draw[dashed] (B) to[out=35, in=145] node[above]{$e_4$} (D);
\draw[dashed] (C) to[out=250, in=270] node[midway,below]{$e_5$} (B);
\draw[dashed] (C) to[out=290, in=270] node[midway,below]{$e_6$} (D);
\end{tikzpicture}
\caption{The modular graph $\widehat{G}_{X_6}$. The tree $T$ is shown by thick lines. The set $E^+$ is shown by dashed lines.}\label{fig: graphX_6}
\end{figure}

Then the following proposition can be easily verified using Lemmas~\ref{vertex class}--\ref{face class}.
\begin{prop}
The data $(T,E^+,F)$ determines a data of representatives for the action of~$\wG_{X_6}$ on~$\wM_{X_6}$.
\end{prop}

Then we get the Brown generators as follows:
\begin{enumerate}\itemsep=0pt
\item The isotropy group of the vertices $v_0$, $v_1$, $v_3$ and $v_4$ are the same, and generated by an involution $\sigma$. It is written as $\sigma=(2\ 4)(3\ 5)$ in $\wG_{Q_0}$ and $\wG_{Q_4}$, $\sigma=(2\ 5)(3\ 4)$ in $\wG_{Q_1}$ and $\wG_{Q_3}$. The isotropy group of the vertex $v_2$ is generated by a permutation $\rho:=(1\ 2\ 5)(3\ 6\ 4) \in \wG_{Q_2}$ of order 3.
\item Set $g_{e^i}:=1$ for $i=0,\dots,3$, and
\begin{gather*}
g_{e_0}:=(1\ 6)\mu_1\mu_6\mu_1\mu_6\mu_1 \in \wG_{v_0},\\
g_{e_1}:=(2\ 3)\mu_2 \in \wG_{v_0}, \\
g_{e_2}:=(2\ 3)\mu_2 \in \wG_{v_4}, \\
g_{e_3}:=(4\ 6)(1\ 3)\mu_2(3\ 5)\mu_6(2\ 4)(1\ 6)\mu_6 \in \wG_{v_3}, \\
g_{e_4}:=(4\ 6)(1\ 2)\mu_3(3\ 5)\mu_6(2\ 4)(1\ 6)\mu_6 \in \wG_{v_3},\\
g_{e_5}:=(3\ 5)(2\ 1\ 5\ 4\ 6\ 3)\mu_4 \in \wG_{v_2}, \\
g_{e_6}:=(2\ 4)(1\ 6)\mu_6(1\ 3\ 6\ 5)(2\ 4)\mu_3 \in \wG_{v_2}.
\end{gather*}
Then they satisfy the condition~(\ref{Brown generator}). Changing their basepoints to the vertex $v_0$ \emph{along} the tree $T$, we get the corresponding elements in~$\wG_{Q_0}$:
\begin{alignat*}{3}
&\phi_0:=(1\ 6)\mu_1\mu_6\mu_1\mu_6\mu_1,\qquad &&
\phi_1:=(2\ 3)\mu_2,&\\
& \phi_2:=(2\ 3)\mu_6\mu_2\mu_6,\qquad &&
\phi_3:=(1\ 5\ 2\ 6\ 3\ 4)\mu_4\mu_2\mu_4\mu_3\mu_1,&\\
&\phi_4:=(1\ 4)(2\ 5)(3\ 6)\mu_4\mu_3\mu_4\mu_2\mu_1,\qquad &&
\phi_5:=(1\ 5\ 3)(2\ 4\ 6)\mu_3\mu_4\mu_6\mu_1,&\\
&\phi_6:=(1\ 5)(2\ 6)\mu_5\mu_2\mu_6\mu_1.&&&
\end{alignat*}
\end{enumerate}
For example, $\phi_3=((2\ 4)(1\ 6)\mu_6(3\ 5)\mu_6(2\ 3)\mu_1)^{-1}g_{e_3}((2\ 4)(1\ 6)\mu_6(3\ 5)\mu_6(2\ 3)\mu_1)$. From Theo\-rem~\ref{thm Brown}, these elements generate the saturated cluster modular group of type~$X_6$. Let us investigate the relations among them.

{\bf Isotropy relations.} We have the relations $\sigma^2=1$, $\rho^3=1$ and $\rho=\sigma \phi_6$.

{\bf Face relations.}
The verification of the following face relations is tedious but straightforward.
\begin{itemize}\itemsep=0pt
\item $\tau_1$: $\phi_2=\phi_1$. This relation implies that $\tau_2$ and $\tau_3$ induces the same relation.
\item $\tau_2(=\tau_3)$: $[\sigma,\phi_3]^2=1$.
\item $\tau_4$: $\sigma^{-1}\phi_4^{-1}\phi_6\phi_4=1$.
\item $\tau_5$: $(\phi_1\sigma)^2\big(\phi_1^{-1}\sigma\big)^2=1$.
\item $\tau_6$: $\big(\sigma\phi_4^{-1}\sigma^{-1}\big)\phi_4\phi_5=1$.
\item $\tau_7$: $\phi_0=1$.
\item $\tau_8$: $\phi_1^{-1}\phi_3^{-1}\phi_4=1$.
\item $\tau_9$: $\big(\sigma\phi_1\sigma^{-1}\big)\phi_3\phi_4^{-1}=1$.
\item $\tau_{10}$: $\phi_4^{-1}\big(\sigma\phi_3\sigma^{-1}\big)\phi_5\phi_3=\sigma$.
\item $\tau_{11}$: $\big(\sigma\phi_4\sigma^{-1}\big)\phi_3^{-1}\big(\sigma\phi_6\sigma^{-1}\big)\phi_3^{-1}=1$.
\end{itemize}
For example, the relation induced by $\tau_4$ is verified as follows. First we change the basepoint to $v_1$ using the path $(2\ 3)\mu_1$ from $v_0$ to $v_1$. Then we have $\phi_4=(1\ 4)(2\ 6)(3\ 5)\mu_2\mu_4\mu_3$ and $\phi_6=(1\ 5)(3\ 6)\mu_3\mu_6$, represented as elements of $\wG_{v_1}$. Then we compute
\begin{align*}
\phi_4^{-1}\phi_6\phi_4&=(1\ 4)(2\ 6)(3\ 5)\mu_5\mu_1\mu_6(1\ 5)(3\ 6)\mu_3\mu_6(1\ 4)(2\ 6)(3\ 5)\mu_2\mu_4\mu_3 \\
&=(2\ 5)(3\ 4)\mu_4\mu_3\mu_4\mu_3=\sigma.
\end{align*}

Now we prove Theorem~\ref{introthm: X_6}.

\begin{proof}[Proof of Theorem~\ref{introthm: X_6}]
From the face relations $\tau_1$, $\tau_4$, $\tau_6$, $\tau_7$ and $\tau_8$, we can reduce the number of generators as $\phi_2=\phi_1$, $\phi_6=\phi_4\sigma\phi_4^{-1}$, $\phi_5=\phi_4^{-1}\sigma\phi_4\sigma=[\phi_4^{-1},\sigma]$, $\phi_0=1$ and $\phi_4=\phi_3\phi_1$. Then the remaining relations are summarized as
\begin{alignat}{3}
&\sigma^2=1, &&&\label{eq:primitive1}\\
&[\sigma,\phi_3\phi_1]^3=1 && (\text{from $\rho^3=1$}),& \label{eq:primitive2}\\
&[\sigma,\phi_3]^2=1 && (\text{from $\tau_2$}),&\label{eq:primitive3} \\
&[\sigma,\phi_1]=\big[\phi_1^{-1},\sigma\big] && (\text{from $\tau_5$}),&\label{eq:primitive4} \\
&(\sigma\phi_1\sigma^{-1})\phi_3(\phi_3\phi_1)^{-1}=1 && (\text{from $\tau_9$}),& \label{eq:primitive5}\\
&(\phi_3\phi_1)^{-1}\big(\sigma\phi_3\sigma^{-1}\big)\big[(\phi_3\phi_1)^{-1},\sigma\big]\phi_3=\sigma \qquad && (\text{from $\tau_{10}$}),& \label{eq:primitive6}\\
&\big(\sigma\phi_3\phi_1\sigma^{-1}\big)\phi_3^{-1}\sigma[\phi_3\phi_1,\sigma]\phi_3^{-1}=1 \qquad && (\text{from $\tau_{11}$}).&\label{eq:primitive7}
\end{alignat}
Using the relations \eqref{eq:primitive1} and \eqref{eq:primitive5} in the form $\phi_3\phi_1=\sigma\phi_1\sigma^{-1}\phi_3$, the last relation is slightly simplified as
\begin{align}
\phi_3=\phi_1[\sigma,\phi_3]\sigma[\phi_3\phi_1,\sigma].\label{eq:primitive7'} \tag{4.7$'$}
\end{align}
Set $\alpha_1:=\phi_1$, $\beta_1:=\phi_3$. Let us introduce the auxiliary elements $\alpha_2:=\sigma \alpha_1 \sigma^{-1}$ and $\beta_2:=\sigma \beta_1 \sigma^{-1}$ to simplify the presentation. Then the relations \eqref{eq:primitive1}--\eqref{eq:primitive5} can be rewritten as follows:
\begin{gather}
 \sigma^2=1, \label{eq:simplified1} \\
 \big(\alpha_2\beta_2\beta_1^{-1}\alpha_1^{-1}\big)^3 = 1,\label{eq:simplified2} \\
 \big(\beta_2\beta_1^{-1}\big)^2 = 1, \label{eq:simplified3}\\
 \alpha_1\alpha_2 = \alpha_2\alpha_1, \label{eq:simplified4}\\
 \beta_1^{-1}\alpha_2\beta_1 = \alpha_1.\label{eq:simplified5}
\end{gather}
The relation \eqref{eq:simplified5} means that $\beta_1$ conjugates $\alpha_1$ to $\alpha_2$. Equivalently, $\beta_2$ conjugates $\alpha_2$ to $\alpha_1$.
The relation~\eqref{eq:primitive6} is equivalently transformed as
\begin{gather*}
 \alpha_1^{-1}\beta_1^{-1}\beta_2\big[(\beta_1\alpha_1)^{-1},\sigma\big]\beta_1 = \sigma, \\
 \alpha_1^{-1}\beta_1^{-1}\beta_2\big(\alpha_1^{-1}\beta_1^{-1}\beta_2\alpha_2\big)\beta_1 = \sigma, \\
 \big(\beta_1\alpha_1^{-1}\beta_1^{-1}\big)\beta_2\alpha_1^{-1}\beta_1^{-1}\beta_2\alpha_2 = \beta_1\sigma \beta_1^{-1}, \\
 \alpha_2^{-1}\beta_2\alpha_1^{-1}\beta_1^{-1}\beta_2\alpha_2 = \beta_1\sigma \beta_1^{-1},
\end{gather*}
where in the last line we used~\eqref{eq:simplified5}. Finally we get $\beta_2 (\beta_1\alpha_1)^{-1}\beta_2 = \operatorname{Ad}_{\alpha_2\beta_1} \sigma$. The relation~\eqref{eq:primitive7'} is equivalent to
\begin{align*}
 \beta_1 &= \alpha_1 [\sigma,\beta_1]\sigma [\beta_1\alpha_1,\sigma]
 =\alpha_1 \big(\beta_2\beta_1^{-1} \beta_2\big)\alpha_2\alpha_1^{-1}\beta_1^{-1}\sigma^{-1} \\
 &=\alpha_1 \beta_1 \alpha_2\big(\alpha_1^{-1}\beta_1^{-1}\big)\sigma^{-1} \qquad \mbox{(use \eqref{eq:simplified3})} \\
 &=\alpha_1 \beta_1 \alpha_2\big(\beta_1^{-1}\alpha_2^{-1}\big)\sigma^{-1}. \qquad \mbox{(use \eqref{eq:simplified5})}
\end{align*}
Thus we get $\beta_1 = \alpha_1 \big(\beta_1 \alpha_2\beta_1^{-1}\big)\alpha_2^{-1}\sigma^{-1}$. Summarizing the relations, we get the desired asser\-tion.
\end{proof}

\section{Generation by cluster Dehn twists}\label{sec:generation}
In this section, we verify that several (saturated) cluster modular groups of finite mutation type are generated by cluster Dehn twists. We will use the notation of~\cite{FeST12} and~\cite{FeST} for quivers of finite mutation type (recall Theorems~\ref{thm:FSTskewsym} and~\ref{thm:FSTgeneral}).

\subsection[Type $X_7$ and $X_6$]{Type $\boldsymbol{X_7}$ and $\boldsymbol{X_6}$}

Recall the presentations given in Theorems~\ref{introthm: X_7} and~\ref{introthm: X_6}.
Let us prove the following:

\begin{prop}\label{cor: Dehn twist} The cluster modular group of type $X_7$ is generated by the six cluster Dehn twists $\phi_k$ and $\psi_k$ $(k=1,3,5)$. The cluster modular group of type $X_6$ has a subgroup of index at most two generated by the four cluster Dehn twists $\alpha_k$ for $k=1,2,3,4$, where $\alpha_3:=\beta_1 \alpha_1 \beta_1^{-1}$ and $\alpha_4:=\sigma \alpha_3 \sigma^{-1}$.
\end{prop}

\begin{proof}It suffices to show that the corresponding saturated cluster modular groups have desired generators, since the cluster modular group is its quotient.

{\bf Type $\boldsymbol{X_7}$.} We can delete the permutations using the relations $\sigma_{35}=\psi_1^{-2}\phi_1$, $\sigma_{51}=\psi_3^{-2}\phi_3$ and $\sigma_{13}=\psi_5^{-2}\phi_5$. Hence the saturated cluster modular group is generated by $\psi_k$'s and $\phi_k$'s. The elements $\phi_k$ are clearly cluster Dehn twists by definition. The elements $\psi_k$ are also cluster Dehn twists, since $\psi_k^4 = \phi_k^2$.

{\bf Type $\boldsymbol{X_6}$.} The element $\alpha_1$ is clearly a cluster Dehn twist by definition. Note that any conjugate of a cluster Dehn twist is also a cluster Dehn twist. In particular, the elements $\alpha_2$, $\alpha_3$ and $\alpha_4$ are cluster Dehn twists.

From the relations
\begin{gather*}
 \beta_1 = \alpha_1 \big(\beta_1 \alpha_2\beta_1^{-1}\big)\alpha_2^{-1}\sigma^{-1}= \alpha_1\alpha_3\alpha_2^{-1}\sigma, \qquad
 \beta_2 = \sigma \beta_1 \sigma^{-1} =\alpha_2\alpha_4\alpha_1^{-1}\sigma,
\end{gather*}
any element $g \in \wG_{X_6}$ can be written as a product of the cluster Dehn twists $\alpha_k$ for $k=1,2,3,4$ and $\sigma$. Since $\sigma^2=1$ and we know the commutation relations of $\sigma$ and $\alpha_k$, we can move all the $\sigma$'s to the right and~$g$ can be written as a product of the cluster Dehn twists $\alpha_k$ for $k=1,2,3,4$ and only one $\sigma$ on the extreme right. Thus the subgroup of $\wG_{X_6}$ generated by the cluster Dehn twists $\alpha_k$ for $k=1,2,3,4$ has index~$2$.
\end{proof}

Note that for the $X_6$ case, it can possibly occur that the whole cluster modular group is generated by $\alpha_k$ for $k=1,2,3,4$ with a help of an additional ``non-standard'' relation involving~$\sigma$.

\subsection[Type $G_2^{(*,*)}$]{Type $\boldsymbol{G_2^{(*,*)}}$}
The weighted quiver representing the mutation class $G_2^{(*,*)}$ is shown in Fig.~\ref{fig: G_2}.
\begin{thm}[Fock--Goncharov {\cite[Corollary~4.2]{FG06a}}]
The saturated cluster modular group of type~$G_2^{(*,*)}$ is isomorphic to the braid group of type $G_2$:
\[
\wG_{G_2^{(*,*)}} \cong \langle a,b \,|\,(ab)^3=(ba)^3 \rangle.
\]
\end{thm}
Although this theorem is obtained by a more topological observation, it can be also verified by a computation based on Section~ref{apply cluster} using the data of representative (which is called the ``spanning tree'') chosen \emph{loc.\ cit.}
\begin{prop}
The generators $a$ and $b$ are given by $a:=p_a((3\ 4)\mu_4)p_a^{-1}$ and $b:=\phi b'\phi^{-1}$, which are cluster Dehn twists.
Here $b':=p_b((1\ 2)\mu_2)p_b^{-1}$ and $\phi:=(3\ 4)\mu_4\mu_3\mu_4\mu_2\mu_1$ are elements of $\Gamma_{G_2^{(*,*)}}$, $p_a:=(1\ 2)(3\ 4)\mu_1$ and $p_b:=(3\ 4)\mu_4\mu_2\mu_1$ are mutation sequences. In particular the cluster modular group of type $G_2^{(*,*)}$ is generated by the two cluster Dehn twists $a$ and~$b$.
\end{prop}

\begin{proof}The elements $a$ and $b'$ are two of the five Brown generators associated with the data of representatives fixed in~\cite{FG06a}. The mutation sequences $p_a^{-1}$ and $p_b^{-1}$ are the ones which connect~$G_2^{(*,*)}$ with the quivers $\bi_a$ and $\bi_b$ shown in Fig.~\ref{fig: equivG_2}, respectively. Then observe that the vertices~3 and~4 are connected by two arrows in the quiver $\bi_a$. Hence the element $(3\ 4)\mu_4 \in \Gamma_{\bi_a}$ is a cluster Dehn twist, so is $a \in \Gamma_{G_2^{(*,*)}}$. Similarly $b'$ is a cluster Dehn twist, so is $b=\phi b'\phi^{-1}$.
\end{proof}

\begin{figure}[t]\centering
\scalebox{0.9}{
\begin{tikzpicture}

\begin{scope}[>=latex]
\path (0,0) node[circle]{$3$} coordinate(A);
\path (2,0) node[circle]{$3$} coordinate(B);
\draw (0,0) circle[radius=0.2]node[above]{$1$};
\draw (2,0) circle[radius=0.2]node[above]{$2$};
\fill (0,-2) circle(2pt) coordinate(C) node[below]{$3$};
\fill (2,-2) circle(2pt) coordinate(D) node[below]{$4$};

\draw[->,shorten >=5pt,shorten <=5pt] (A) -- (B) [thick];
\draw[->,shorten >=5pt,shorten <=2pt] (C) -- (A) [thick];
\draw[->,shorten >=2pt,shorten <=5pt] (B) -- (C) [thick];
\draw[->,shorten >=5pt,shorten <=2pt] (D) -- (B) [thick];
\draw[->,shorten >=2pt,shorten <=2pt] (C) -- (D) [thick];

\end{scope}
\end{tikzpicture}
}
\caption{The quiver of type $G_2^{(*,*)}$.}\label{fig: G_2}
\end{figure}

\begin{figure}[t!]\centering
\scalebox{0.9}{
\begin{tikzpicture}
\begin{scope}[>=latex]
\path (0,0) node[circle]{$3$} coordinate(A);
\path (2,0) node[circle]{$3$} coordinate(B);
\draw (0,0) circle[radius=0.2]node[above]{$1$};
\draw (2,0) circle[radius=0.2]node[above]{$2$};
\fill (0,-2) circle(2pt) coordinate(C) node[below]{$3$};
\fill (2,-2) circle(2pt) coordinate(D) node[below]{$4$};
\draw(1,-3) node{$Q_a$};

\draw[->,shorten >=5pt,shorten <=5pt] (A) -- (B) [thick];
\draw[->,shorten >=2pt,shorten <=5pt] (A) -- (C) [thick];
\draw[->,shorten >=5pt,shorten <=2pt] (D) -- (A) [thick];
\draw[->,double,shorten >=2pt,shorten <=2pt] (C) -- (D) [thick];

\path (5,0) node[circle]{$3$} coordinate(A1);
\path (7,0) node[circle]{$3$} coordinate(B1);
\draw (5,0) circle[radius=0.2]node[above]{$1$};
\draw (7,0) circle[radius=0.2]node[above]{$2$};
\fill (5,-2) circle(2pt) coordinate(C1) node[below]{$3$};
\fill (7,-2) circle(2pt) coordinate(D1) node[below]{$4$};
\draw(6,-3) node{$Q_b$};

\draw[->,shorten >=5pt,shorten <=2pt] (C1) -- (B1) [thick];
\draw[->,shorten >=2pt,shorten <=5pt] (A1) -- (C1) [thick];
\draw[->,shorten >=2pt,shorten <=2pt] (C1) -- (D1) [thick];
\draw[->,double,shorten >=5pt,shorten <=5pt] (B1) -- (A1) [thick];
\end{scope}
\end{tikzpicture}}
\caption{Two quivers equivalent to $G_2^{(*,*)}$.}\label{fig: equivG_2}
\end{figure}

\subsection[Type $\tilde{E}_6$, $\tilde{E}_7$ and $\tilde{E}_8$]{Type $\boldsymbol{\tilde{E}_6}$, $\boldsymbol{\tilde{E}_7}$ and $\boldsymbol{\tilde{E}_8}$}

In~\cite{ASS}, they introduce the \emph{cluster automorphism groups} and compute them for quivers of finite or euclidean types using the cluster categories. It is known that if a quiver has no frozen vertices, then the cluster modular group is isomorphic to the \emph{group of direct cluster automorphisms} $\operatorname{Aut}^+(\A)$ (here $\A$ stands for the \emph{cluster algebra}), which is a subgroup of index $\leq 2$ of the cluster automorphism group. See~\cite{Fraser}. Hence we can rephrase the results shown in~\cite[Table~1]{ASS} as follows, restricting our attention to the quivers of affine and finite mutation type.

\begin{thm}[Assem--Schiffler--Shramchenko~\cite{ASS}]
The cluster modular groups of affine type $\tilde{E}_6$, $\tilde{E}_7$ and $\tilde{E}_8$ are given as follows:
\begin{itemize}\itemsep=0pt
\item $\Gamma_{\tilde{E}_6} \cong \mathbb{Z}\times \mathfrak{S}_3$.
\item $\Gamma_{\tilde{E}_7} \cong \mathbb{Z}\times\mathbb{Z}/2$.
\item $\Gamma_{\tilde{E}_8} \cong \mathbb{Z}$.
\end{itemize}
\end{thm}

\begin{prop}The cluster modular groups of type $\tilde{E}_6$, $\tilde{E}_7$ and $\tilde{E}_8$ are generated by cluster Dehn twists.
\end{prop}

\begin{proof}
Since each group is virtually cyclic, it suffices to find one cluster Dehn twist $\phi$. Indeed, a sufficiently large power of another generator $\psi$ of infinite order lies in the cyclic group generated by $\phi$ and hence $\psi$ is also a cluster Dehn twist from the definition.

One can directly verify that the initial quiver of type $\tilde{E}_6$ is mutation-equivalent to the quiver shown in Fig.~\ref{fig: E_6}. Then we have two arrows from the vertex $4$ to $3$ in $\bi$, and other vertices $x$ are either disjoint from these vertices or connected in the form $3 \to x \to 4$.
Hence by the proof of \cite[Lemma~2.32]{Ish}, $\phi:=(3\ 4)\mu_3 \in \Gamma_\bi$ is a cluster Dehn twist. In each of the mutation classes $\tilde{E}_7$ and $\tilde{E}_8$, one can similarly find a quiver with a pair of vertices $u$, $v$ and two arrows from $v$ to $u$, such that the other vertices $x$ are either disjoint from them or connected in the form $u \to x \to v$. Hence we get a cluster Dehn twist.
\end{proof}

\begin{figure}[t!]\centering
\begin{tikzpicture}
\begin{scope}[>=latex]
\fill (0,0) circle(2pt) coordinate(A) node[below]{$1$};
\fill (1,0) circle(2pt) coordinate(B) node[below]{$2$};
\fill (2,1) circle(2pt) coordinate(C) node[above]{$3$};
\fill (2,-1) circle(2pt) coordinate(D) node[below]{$4$};
\fill (3,0) circle(2pt) coordinate(E) node[below]{$5$};
\fill (4,0) circle(2pt) coordinate(F) node[below]{$6$};
\fill (5,0) circle(2pt) coordinate(G) node[below]{$7$};

\draw[->,shorten >=2pt,shorten <=2pt] (B) -- (A) [thick];
\draw[->,shorten >=2pt,shorten <=2pt] (C) -- (B) [thick];
\draw[->,shorten >=2pt,shorten <=2pt] (B) -- (D) [thick];
\draw[->,shorten >=2pt,shorten <=2pt] (C) -- (E) [thick];
\draw[->,shorten >=2pt,shorten <=2pt] (E) -- (D) [thick];
\draw[->,shorten >=2pt,shorten <=2pt] (C) -- (F) [thick];
\draw[->,shorten >=2pt,shorten <=2pt] (F) -- (D) [thick];
\draw[->,shorten >=2pt,shorten <=2pt] (F) -- (G) [thick];
\draw[->,double,shorten >=2pt,shorten <=2pt] (D) -- (C) [thick];
\end{scope}
\end{tikzpicture}
\caption{A quiver equivalent to $\tilde{E}_6$.}\label{fig: E_6}
\end{figure}

\appendix

\section{Relations with the mapping class group of an annulus}\label{mapping class groups}
In this section, we investigate several relations between the cluster modular group of type~$X_7$ with the mapping class group of an annulus. Recall from Example~\ref{ex:MCG} and Theorem~\ref{thm:FSTskewsym} that an ideal triangulation~$\Delta$ of a marked surface $\Sigma$ gives rise to the seed $\bi_\Delta$ of finite mutation type. The following lemma is repeatedly used in this section.
\begin{prop}\label{elimination}
For a seed $\bi=(\ve,(A_i)_{i \in I}, (X_i)_{i \in I})$, define its unfrozen part to be
\begin{gather*}
 \bi_\uf:=\big((\ve_{ij})_{i,j \in I \setminus I_0},(A_i)_{i \in I \setminus I_0}, (X_i)_{i \in I \setminus I_0}\big).
\end{gather*}
Then we have a group homomorphism $e\colon \Gamma_{\bi} \to \Gamma_{\bi_\uf}$ which fits into the following exact sequence:
\[
1 \to \operatorname{Aut}_0(\bi) \to \Gamma_{\bi} \xrightarrow{e} \Gamma_{\bi_\uf},
\]
where $\operatorname{Aut}_0(\bi):=\{ \sigma \in \operatorname{Aut}(\bi)\,|\, \text{$\sigma(i)=i$ for all $i \in I-I_0$}\}$.
\end{prop}

\begin{figure}[t!]\centering
\begin{tikzpicture}
\begin{scope}[>=latex]
\fill (0,0) circle(2pt) coordinate(A) node[below]{$0$};
\fill (A) ++(120: 2) circle(2pt) coordinate(B) node[left]{$1$};
\fill (A) ++(60: 2) circle(2pt) coordinate(C) node[right]{$2$};

\draw[->,shorten >=2pt,shorten <=2pt] (A) -- (B) [thick];
\draw[->,double,shorten >=2pt,shorten <=2pt] (B) -- (C) [thick];
\draw[->,shorten >=2pt,shorten <=2pt] (C) -- (A) [thick];

\end{scope}
\end{tikzpicture}
\caption{The quiver $\bj$.}\label{fig: j}
\end{figure}

\begin{figure}[t!]\centering
\begin{tikzpicture}
\fill[gray] (0,0) circle [radius=0.5];
\draw (0,0) circle [radius=2];
\fill(0.5,0) circle(2pt) coordinate(A);
\fill(-0.5,0) circle(2pt) coordinate(B);
\coordinate(B') at (-1,0);
\coordinate(B+) at (0,1);
\coordinate(B-) at (0,-1);
\fill(2,0) circle(2pt) coordinate(C);
\coordinate(D') at (-1.5,0);
\coordinate(D+) at (0,1.5);
\coordinate(D-) at (0,-1.5);
\draw (A) to[out=90, in=0] (B+) to[out=180, in=90] (B') to[out=270, in=180] (B-) node[above]{0} to[out=0, in=270] (A) ;
\draw (A) -- node[above]{2} (C) ;
\draw (A) to[out=60, in=0] node[above]{1} (D+) to[out=180, in=90] (D') to[out=270, in=180] (D-) to[out=0, in=210] (C);
\end{tikzpicture}
\caption{The ideal triangulation $\Delta$ of the annulus with (2+1) marked points.}
\label{fig: annulus}
\end{figure}

We call the homomorphism $e\colon \Gamma_{\bi} \to \Gamma_{\bi_\uf}$ the \emph{elimination homomorphism}.
\begin{proof}Each element $\phi \in \Gamma_\bi$ can be written as $\phi=\seq$, where $i_1,\dots,i_k \in I-I_0$ and $\sigma$ is a seed isomorphism. Here we can write $\sigma=\sigma_0\sigma_\uf$, where $\sigma_0$ is a permutation of $I_0$ and $\sigma_\uf$ is a permutation of $I-I_0$. Then we define $e(\phi):=\sigma_\uf\mu_{i_l}\cdots\mu_{i_1} \in \Gamma_{\bi_\uf}$. The map $e$ is a group homomorphism, since the permutation $\sigma_0$ commutes with mutations. Moreover $e(\phi)=1$ if and only if $\phi=\sigma_0 \in \operatorname{Aut}_0(\bi)$. Thus we get the desired exact sequence.
\end{proof}

Similarly we can define the elimination homomorphism between the corresponding saturated cluster modular groups, and we have a similar exact sequence. The elimination homomorphism is not surjective in general, and the image $e(\Gamma_\bs) \subset \Gamma_{\bi_\uf}$ can have infinite index.

Let $\bj$ be the quiver shown in Fig.~\ref{fig: j}. Note that it gives a subquiver of both $X_6$ and $X_7$.
\begin{lem}
The quiver $\bj$ coincides with the unfrozen part of the quiver associated with the ideal triangulation $\Delta$ of the annulus $S$ with $(1+2)$ marked points shown in Fig.~{\rm \ref{fig: annulus}}. Moreover, the elimination homomorphism induces an isomorphism $e\colon \Gamma_{\bi_\Delta}={\rm MC}(S) \xrightarrow{\sim} \Gamma_{\bj}$. Under this isomorphism, the generator $s \in {\rm MC}(S) \cong \mathbb{Z}$ which rotates the two marked points on one boundary component $($called a ``half twist''$)$ corresponds to the element $(0\ 1\ 2)\mu_2\mu_1\mu_0 \in \Gamma_{\bj}$. The Dehn twist $t:=s^2 \in {\rm MC}(S)$ corresponds to the element $(1\ 2)\mu_1 \in \Gamma_{\bj}$.
\end{lem}

\begin{proof}The injectivity is easily verified using the exact sequence in Proposition~\ref{elimination}. It can be verified that $\Gamma_{\bj}$ is generated by the element $(0\ 1\ 2)\mu_2\mu_1\mu_0$, by observing that the mutation class $|\bj|$ consists of two seeds and the element acts on the set $\operatorname{Mat}^{-1}(\bj) \subset V\big(\wM_{|\bj|}\big)$ transitively. Then the surjectivity follows from the third statement. The verifications of the third and fourth statements are straightforward.
\end{proof}

\begin{rem}
We have $\wG_{\bi_\Delta}\cong \Gamma_{\bi_\Delta}$ in this case, see Theorem 9.17 in~\cite{FST08}. Hence we also have $\wG_{\bj} \cong \Gamma_{\bj}$.
\end{rem}

Let $\Gamma_{X_7}^{(12)}$ be the subgroup of $\Gamma_{X_7}$ which consists of elements of the form $\phi=\seq$, where $i_1,\dots, i_k \in \{0,1,2\}$ and $\sigma$ is a permutation such that $\sigma(\{0,1,2\})=\{0,1,2\}$. Note that the group $\Gamma_{X_7}^{(12)}$ is naturally identified with the cluster modular group of the seed obtained from the seed $X_7$ by freezing the vertices $\{3,4,5,6\}$.

\begin{lem}
We have a split exact sequence
\[
1 \to \mathbb{Z}/2 \xrightarrow{i^{(12)}} \Gamma_{X_7}^{(12)} \xrightarrow{e^{(12)}} \Gamma_{\bj}={\rm MC}(S) \to 1.
\]Here $i^{(12)}$ is given by $1 \mapsto \sigma_{35}=(3\ 5)(4\ 6)$ and $e^{(12)}$ is the elimination homomorphism.
\end{lem}

\begin{proof}The exactness of the sequence follows from Proposition~\ref{elimination} and $s=e^{(12)}(\psi_1)$. The commutativity relations $\psi_1\sigma_{35}=\sigma_{35}\psi_1$ and $\phi_1\sigma_{35}=\sigma_{35}\phi_1$ in Theorem~\ref{introthm: X_7} show that the image of $i^{(12)}$ is central. Since ${\rm MC}(S)\cong \mathbb{Z}$, there exists a section of~$e^{(12)}$.
\end{proof}

Having this lemma in mind, we write ${\rm MC}(S)^{(12)}:=\Gamma_{X_7}^{(12)} \cong {\rm MC}(S) \times \mathbb{Z}/2$.
The relation $t=s^2$ in ${\rm MC}(S)$ corresponds to the relation $\phi_1=\psi_1^2\sigma_{35}$ in ${\rm MC}(S)^{(12)}$.

Similarly we define the subgroups ${\rm MC}(S)^{(34)}$, ${\rm MC}(S)^{(56)}$ of the cluster modular group $\Gamma_{X_7}$. They are isomorphic to ${\rm MC}(S)\times \mathbb{Z}/2$. Summarizing, we have proved the following geometric interpretation of Proposition~\ref{cor: Dehn twist} in terms of these extended mapping class groups:

\begin{thm}The extended mapping class groups ${\rm MC}(S)^{(12)}$, ${\rm MC}(S)^{(34)}$ and ${\rm MC}(S)^{(56)}$ ge\-ne\-rate the cluster modular group $\Gamma_{X_7}$. The six cluster Dehn twists given in Proposition~{\rm \ref{cor: Dehn twist}} corresponds to the half twist and Dehn twist in ${\rm MC}(S)$ via the elimination homomorphisms.
\end{thm}

\subsection*{Acknowledgements}
The author would like to express his gratitude to his supervisor, Nariya Kawazumi for continuous encouragement during this work. He is very grateful to Travis Scrimshaw for pointing out an error in the presentation of $\wG_{X_6}$ in the first version of this paper. Most of computations in this paper are done by using the \emph{Java applet for quiver mutations} provided by Bernhard Keller, which is available at \url{http://www.math.lsa.umich.edu/~fomin/cluster.html}. This work is partially supported by JSPS KAKENHI Grant Number 18J13304 and the program for Leading Graduate School, MEXT, Japan.

\pdfbookmark[1]{References}{ref}
\LastPageEnding


\begin{thebibliography}{99}
\footnotesize\itemsep=0pt

\bibitem{ASS}
Assem I., Schiffler R., Shramchenko V., Cluster automorphisms, \href{https://doi.org/10.1112/plms/pdr049}{\textit{Proc.
 Lond. Math. Soc.}} \textbf{104} (2012), 1271--1302, \href{https://arxiv.org/abs/1009.0742}{arXiv:1009.0742}.

\bibitem{BS15}
Bridgeland T., Smith I., Quadratic differentials as stability conditions,
 \href{https://doi.org/10.1007/s10240-014-0066-5}{\textit{Publ. Math. Inst. Hautes \'Etudes Sci.}} \textbf{121} (2015),
 155--278, \href{https://arxiv.org/abs/1302.7030}{arXiv:1302.7030}.

\bibitem{Brown84}
Brown K.S., Presentations for groups acting on simply-connected complexes,
 \href{https://doi.org/10.1016/0022-4049(84)90009-4}{\textit{J.~Pure Appl. Algebra}} \textbf{32} (1984), 1--10.

\bibitem{FM}
Farb B., Margalit D., A primer on mapping class groups, \textit{Princeton
 Mathematical Series}, Vol.~49, \href{https://doi.org/10.23943/princeton/9780691147949.001.0001}{Princeton University Press}, Princeton, NJ, 2012.

\bibitem{FeST}
Felikson A., Shapiro M., Tumarkin P., Cluster algebras of finite mutation type
 via unfoldings, \href{https://doi.org/10.1093/imrn/rnr072}{\textit{Int. Math. Res. Not.}} \textbf{2012} (2012),
 1768--1804, \href{https://arxiv.org/abs/1006.4276}{arXiv:1006.4276}.

\bibitem{FeST12}
Felikson A., Shapiro M., Tumarkin P., Skew-symmetric cluster algebras of finite
 mutation type, \href{https://doi.org/10.4171/JEMS/329}{\textit{J.~Eur. Math. Soc.}} \textbf{14} (2012), 1135--1180,
 \href{https://arxiv.org/abs/0811.1703}{arXiv:0811.1703}.

\bibitem{FG06a}
Fock V.V., Goncharov A.B., Cluster {${\mathcal X}$}-varieties, amalgamation,
 and {P}oisson--{L}ie groups, in Algebraic Geometry and Number Theory,
 \textit{Progr. Math.}, Vol.~253, \href{https://doi.org/10.1007/978-0-8176-4532-8_2}{Birkh\"{a}user Boston}, Boston, MA, 2006,
 27--68, \href{https://arxiv.org/abs/math.RT/0508408}{arXiv:math.RT/0508408}.

\bibitem{FG06}
Fock V.V., Goncharov A.B., Moduli spaces of local systems and higher
 {T}eichm\"{u}ller theory, \href{https://doi.org/10.1007/s10240-006-0039-4}{\textit{Publ. Math. Inst. Hautes \'Etudes Sci.}}
 (2006), 1--211, \href{https://arxiv.org/abs/math.AG/0311149}{arXiv:math.AG/0311149}.

\bibitem{FG09}
Fock V.V., Goncharov A.B., Cluster ensembles, quantization and the dilogarithm,
 \href{https://doi.org/10.24033/asens.2112}{\textit{Ann. Sci. \'Ec. Norm. Sup\'er.~(4)}} \textbf{42} (2009), 865--930,
 \href{https://arxiv.org/abs/math.AG/0311245}{arXiv:math.AG/0311245}.

\bibitem{FG08}
Fock V.V., Goncharov A.B., The quantum dilogarithm and representations of
 quantum cluster varieties, \href{https://doi.org/10.1007/s00222-008-0149-3}{\textit{Invent. Math.}} \textbf{175} (2009),
 223--286, \href{https://arxiv.org/abs/math.QA/0702397}{arXiv:math.QA/0702397}.

\bibitem{FST08}
Fomin S., Shapiro M., Thurston D., Cluster algebras and triangulated surfaces.
 {I}.~{C}luster complexes, \href{https://doi.org/10.1007/s11511-008-0030-7}{\textit{Acta Math.}} \textbf{201} (2008), 83--146,
 \href{https://arxiv.org/abs/math.RA/0608367}{arXiv:math.RA/0608367}.

\bibitem{FZ03}
Fomin S., Zelevinsky A., Cluster algebras. {II}.~{F}inite type classification,
 \href{https://doi.org/10.1007/s00222-003-0302-y}{\textit{Invent. Math.}} \textbf{154} (2003), 63--121,
 \href{https://arxiv.org/abs/math.RA/0208229}{arXiv:math.RA/0208229}.

\bibitem{FZ07}
Fomin S., Zelevinsky A., Cluster algebras. {IV}.~{C}oefficients,
 \href{https://doi.org/10.1112/S0010437X06002521}{\textit{Compos. Math.}} \textbf{143} (2007), 112--164,
 \href{https://arxiv.org/abs/math.RA/0602259}{arXiv:math.RA/0602259}.

\bibitem{Fraser}
Fraser C., Quasi-homomorphisms of cluster algebras, \href{https://doi.org/10.1016/j.aam.2016.06.005}{\textit{Adv. in Appl.
 Math.}} \textbf{81} (2016), 40--77, \href{https://arxiv.org/abs/1509.05385}{arXiv:1509.05385}.

\bibitem{Ish}
Ishibashi T., On a {N}ielsen--{T}hurston classification theory for cluster
 modular groups, \href{https://doi.org/10.5802/aif.3250}{\textit{Ann. Inst. Fourier (Grenoble)}} \textbf{69} (2019),
 515--560, \href{https://arxiv.org/abs/1704.06586}{arXiv:1704.06586}.

\bibitem{KT15}
Kato A., Terashima Y., Quiver mutation loops and partition {$q$}-series,
 \href{https://doi.org/10.1007/s00220-014-2224-5}{\textit{Comm. Math. Phys.}} \textbf{336} (2015), 811--830, \href{https://arxiv.org/abs/1403.6569}{arXiv:1403.6569}.

\bibitem{Kim-Yamazaki}
Kim H.K., Yamazaki M., Comments on exchange graphs in cluster algebras,
 \href{https://doi.org/10.1080/10586458.2018.1437849}{\textit{Exp. Math.}} \textbf{29} (2020), 79--100, \href{https://arxiv.org/abs/1612.00145}{arXiv:1612.00145}.

\bibitem{Penner}
Penner R.C., Decorated {T}eichm\"{u}ller theory, \textit{QGM Master Class Series},
 \href{https://doi.org/10.4171/075}{European Mathematical Society (EMS)}, Z\"urich, 2012.

\bibitem{Zickert}
Zickert C.K., Fock--{G}oncharov coordinates for rank two {L}ie groups,
 \href{https://doi.org/10.1007/s00209-019-02307-8}{\textit{Math.~Z.}} \textbf{294} (2020), 251--286, \href{https://arxiv.org/abs/1605.08297}{arXiv:1605.08297}.

\end{thebibliography}
\end{document}